\RequirePackage{fix-cm}
\documentclass[smallextended]{svjour3}       
\usepackage{graphicx}
\usepackage{amsmath}
\usepackage{amsfonts}
\usepackage{colonequals}
\usepackage{placeins}
\usepackage{color}
\usepackage{mathptmx}      
\usepackage{hyperref}
\usepackage[misc]{ifsym}
\numberwithin{equation}{section}
\numberwithin{theorem}{section}

\journalname{Applied Mathematics and Optimization}

\begin{document}

\title{Accelerating two projection methods via perturbations with application to intensity-modulated radiation therapy}


\author{Esther Bonacker \and
         Aviv Gibali \and Karl-Heinz K\"{u}fer
}


\institute{Esther Bonacker \at
              Optimization Department, Fraunhofer ITWM, 67663 Kaiserslautern, Germany\\
              \email{bonacker@itwm.fraunhofer.de}
           \and
           \Letter Aviv Gibali \at
              Department of Mathematics, ORT Braude College, Karmiel 2161002, Israel\\
              The Center for Mathematics and Scientific Computation, University of Haifa, Mt. Carmel, Haifa 3498838, Israel.\\
              \email{avivg@braude.ac.il}
           \and
           Karl-Heinz K\"{u}fer \at
           Optimization Department,  Fraunhofer ITWM, 67663 Kaiserslautern, Germany\\
           \email{kuefer@itwm.fraunhofer.de}
}

\date{Received: date / Accepted: date}

\maketitle

\begin{abstract}

Constrained convex optimization problems arise naturally in many real-world applications. One strategy to solve them in an approximate way is to translate them into a sequence of convex feasibility problems via the recently developed level set scheme and then solve each feasibility problem using projection methods. However, if the problem is ill-conditioned, projection methods often show zigzagging behavior and therefore converge slowly.	

To address this issue, we exploit the bounded perturbation resilience of the projection methods and introduce two new perturbations which avoid zigzagging behavior. The first perturbation is in the spirit of $k$-step methods and uses gradient information from previous iterates. The second uses the approach of surrogate constraint methods combined with relaxed, averaged projections.

We apply two different projection methods in the unperturbed version, as well as the two perturbed versions, to linear feasibility problems along with nonlinear optimization problems arising from intensity-modulated
radiation therapy (IMRT) treatment planning. We demonstrate that for all the considered problems the perturbations can significantly accelerate the convergence of the projection methods and hence the overall procedure of the level set
scheme. For the IMRT optimization problems the perturbed projection methods found an approximate solution up to 4 times faster than the unperturbed methods while at the same time achieving objective function values which were 0.5 to 5.1\% lower.

\keywords{Projection methods \and subgradient \and bounded perturbation resilient \and inertial methods \and surrogate constraint method}
\subclass{65K10 \and 65K15 \and 90C25 \and 90C90}
\end{abstract}

\section{Introduction}\label{sec:Intro}


The problem we consider in this paper is the \textit{general constrained convex optimization problem}
\begin{align}\label{eq:GeneralOptProblem}
\text{Minimize } &f(x)\\
\text{s.t. }&g_j(x) \leq 0 \quad  j=1,...,m \nonumber
\end{align}
where $f:\mathbb{R}^{n}\rightarrow\mathbb{R}$ and $g_{j}:\mathbb{R}^{n}\rightarrow\mathbb{R}$ for $j\in J = \{1,\ldots,m\}$ are convex functions.

In order to solve this problem we translate it to an equivalent epigraph form \eqref{eq:CFPepigraph} and then use a procedure called the \textit{level set scheme} introduced in \cite{gkrs18} to construct a sequence of convex feasibility problems from it.
In this procedure, we define an upper bound on the function $f$ and thereby generate an additional constraint. This constraint together with the constraints given in \eqref{eq:GeneralOptProblem} formulates a feasibility problem. Each time the feasibility problem can be solved by a suitable algorithmic operator $\boldsymbol{T}$, we reduce the upper bound on $f$ and use it to formulate the next feasibility problem.

The solution $x^{\ast}_{\varepsilon}$ produced by the level set scheme is an approximation of the solution $x^{\ast}$ of the original optimization problem \eqref{eq:GeneralOptProblem} in the sense that $f(x^{\ast}_{\varepsilon}) < f(x^{\ast}) + \varepsilon_s$ for some $\varepsilon_s > 0$, while at the same time both $x^{\ast}_{\varepsilon}$ and $x^{\ast}$ fulfill $g_j(x)\leq 0$ for all $j\in J$. $x^{\ast}_{\varepsilon}$ is therefore called an $\varepsilon_s$-optimal solution of \eqref{eq:GeneralOptProblem}.


In this work we choose the \textit{simultaneous subgradient projection} method and the \textit{cyclic subgradient projection} method as options for the operator $T$ used to solve each feasibility problem and compare their behavior. Both methods, as well as many others suitable to solve a convex feasibility problem, are bounded perturbation resilient according to Definition \ref{eq:DefBPR} \cite{special-issue}. This means that certain changes can be made at each iteration step and $T$ still produces a solution of the feasibility problem. One favorable property of bounded perturbation resilient algorithms is the fact that the convergence rate of the unperturbed variant is preserved by the perturbed algorithm (see e.g. \cite[Theorem 11 and Corollary 14]{brz17}).


There exist two kinds of perturbations, which differ in the way the changes are applied. They are called inner and outer perturbations. Combettes in \cite{combettes01} studied the concept of perturbation resilience of projection methods using inner perturbations, which they and others refer to as convergence under summable errors \cite{tw13}.
Algorithms using outer perturbations are also referred to as inexact methods \cite{sv12}, because the perturbation vectors can be interpreted e.g. as calculation error if the evaluation of $\boldsymbol{T}(x)$ cannot be performed in an exact way. These kind of perturbations are used e.g. in \cite{dgj18}.


One way to exploit bounded perturbation resilience is implemented by the \textit{superiorization methodology}, which uses perturbations pro-actively during the performance of the iteration scheme to steer the algorithm to an output that still solves the given problem, but is superior with respect to a given secondary criterion. See e.g. \cite{censor18,censor15,censor17,cdh10,cdhst14,cz15}.
This approach has been used in  the same application as the one we consider in this paper: Intensity-modulated radiation therapy (IMRT) treatment planning \cite{bgks17,gkrs18}.
Gibali et al. in \cite{gkrs18} presented a detailed comparison of superiorized and unsuperiorized projection methods used in the level set scheme to solve optimization problems in the field of IMRT. In \cite{bgks17} these optimization problems are reformulated as lexicographic optimization problems. The perturbations used in a given level of the lexicographic optimization scheme are chosen to be descent directions with respect to the objective function of the subsequent optimization level. This approach exhibits faster convergence than the unperturbed algorithm.


Our motivation for using perturbations, however, is a different one. When we solve IMRT optimization problems as described previously, we observe zigzagging behavior of the projection methods. This behavior is not surprising as we attempt to satisfy two main goals, which conflict with each other: Irradiating the tumor and at the same time sparing surrounding healthy organs. The zigzagging behavior is similar to the one known to occur in projected gradient methods, which arises from the conflict between seeking feasibility and reduction of the objective function, see \cite{wang2008} and the references therein.

In this paper we introduce the \textit{heavy ball} and the \textit{surrogate constraint} perturbation, which address the zigzagging behavior. We offer formulations of them both as inner and outer perturbations. We show that all of these formulations describe bounded perturbations. Finally, we use the heavy ball and surrogate constraint perturbation together with both the cyclic and simultaneous subgradient projection method.


The heavy ball perturbation uses the approach of the heavy ball method, which belongs to the group of $k$-step or inertial methods. These algorithms use gradient information from $k$ previous iteration steps to calculate their next iterate. Incorporating previous gradient information alleviates zigzagging behavior compared to methods, which only use current gradient information.
In recent years authors demonstrated that many optimization algorithms of this kind can be seen as a discretization of the trajectories of differential equations derived from the field of continuous dynamical systems, e.g., \cite{acpr:18,frv:18,Nesterov:83,Polyak:64}. It is shown that inertial approaches are very effective and demonstrate good convergence properties.


The second perturbation we discuss is the surrogate constraint perturbation. It is inspired by the algorithm presented in \cite{Dudek2007}. This algorithm combines the approach of classical surrogate constraint methods (see e.g. \cite{Yang1992}) with relaxed, averaged projections. The author considers linear feasibility problems and demonstrates that especially for flat solution sets (which are associated with zigzagging behavior by methods using gradient information), their algorithm converges faster than a classical surrogate constraint method. In our work we illustrate that their idea can be successfully translated to the nonlinear context of IMRT optimization problems.


The paper is organized as follows. In Section \ref{sec:Prelim} we present definitions of methods and concepts, which we will use in our algorithm. In Section \ref{sec:Pert} we present some results on the convergence of algorithmic operators using inner and outer perturbations. Then we introduce the perturbations developed by us and show that they are bounded. We end this section by summarizing our algorithm, which incorporates all of the concepts described before. Finally, in Section \ref{sec:Num} we present numerical results from both linear feasibility problems and nonlinear optimization problems arising from IMRT treatment planning.

\section{Preliminaries}\label{sec:Prelim}
We now present concepts and methods as well as mathematical background which is essential for the introduction of our results.

\medskip

In this paper we consider the general constrained convex optimization problem \eqref{eq:GeneralOptProblem}. We reformulate it into an equivalent \textit{epigraph form}

\begin{align}\label{eq:CFPepigraph}
\text{Minimize } t &\in\mathbb{R}\\
\text{s.t. }f(x)-t&\leq 0\nonumber\\
g_j(x) &\leq 0 \quad j=1,...,m.\nonumber
\end{align}

We denote the optimal value of \eqref{eq:CFPepigraph} by $t^{\ast}$. We assume that $t^{\ast}$ is finite and that there exists a point $x^{\ast}$ with $f(x^{\ast}) = t^{\ast}$, which is feasible for \eqref{eq:CFPepigraph}.

Our approach to solving \eqref{eq:CFPepigraph} is to transform it into a sequence of \textit{convex feasibility problems (CFP)}, see also \cite[Subsection 2.1.2]{Bertsekas99}. The general formulation of a CFP is the following:

\begin{align}\label{eq:DefCFP}
\text{Find } x\in C \colonequals \bigcap_{i\in I}C_{i} \colonequals \left\{  x\in\mathbb{R}^{n}\mid \varphi_{i}(x)\leq 0\right\}
\end{align}
where the functions $\varphi_i$ are convex for all $i\in I$.

In the following we describe the procedure we use to solve \eqref{eq:CFPepigraph}, which is called the \textit{level set scheme} \cite{gkrs18}. We use a decreasing sequence $\left\{t_{s}\right\}_{s=1}^{\infty}$.
This sequence is not known a priori. The first element $t_1 = \infty$ is fixed and every subsequent element of the sequence is determined later during the iteration process. We define the functions

\begin{align}
\varphi_1^s &\colonequals  f(x)-t_{s}, \varphi_{j+1}^s \colonequals g_j(x) \quad j=1,...,m.
\end{align}

Then we try to solve the CFP $\boldsymbol{P}^s$:

\begin{align} \label{Problem:k_CFP}
\text{Find } x \in C = \bigcap_{i\in I}C_{i}^s \colonequals \left\{  x\in\mathbb{R}^{n}\mid \varphi_{i}^s(x)\leq 0\right\}
\end{align}
with $I = \{1,..., m+1\}$. If $C\neq \emptyset$, a solution of $\boldsymbol{P}^s$ found by a suitable algorithmic operator $\boldsymbol{T}$ is denoted by $x^{\ast}_s$. Now $t_{s+1}$ is calculated according to a user defined update rule like for example $t_{s+1}=f(x^{\ast}_s)-\varepsilon_{s}$ or $t_{s+1}=f(x^{\ast}_s)(1-\varepsilon_{s})$ where $\{\varepsilon_{s}\}$ is some user chosen sequence with $\varepsilon_{s} > 0$ for all $s\geq 0$. Using $t_{s+1}$ the next CFP $\boldsymbol{P}^{s+1}$ can be formulated and the iteration proceeds.

If $C = \emptyset$ or the algorithmic operator $\boldsymbol{T}$ is unable to find a solution of  $\boldsymbol{P}^s$ for other reasons, we consider $x^{\ast}_{s-1}$ to be the result of the level set scheme.

One successful class of algorithmic operators for solving CFPs is the class of projection methods. These are iterative algorithms that use projections onto sets, relying on the principle that when a family of sets is present, then projections onto the given individual sets are easier to perform than projections onto other sets (intersections, image sets under some transformation, etc.) that are derived from the given individual sets. Projection methods can have various algorithmic structures some of which are
particularly suitable for parallel computing, and possess desirable convergence properties. See e.g. \cite{cz97,combettes97,znly18}.
A special case of a CFP is the linear feasibility problem $Ax = b$. An illustration of different projection methods for this problem can be found in \cite{ceh2000}.

In this paper we present two different projection methods as options for the algorithmic operator $\boldsymbol{T}$. The first one is the cyclic subgradient projection method, which is also used in combination with the level set scheme in \cite{gkrs18}. Like some other projection methods, it utilizes the concept of control sequences. These are sequences $\{i(\nu)\}_{\nu=0}^\infty$, which determine an ordering of the indices of the sets $C_{i}$ mentioned in \eqref{eq:DefCFP}. The kind of control sequence used by the cyclic subgradient projection method is defined below.

\begin{definition}\label{eq:control}
	The sequence $\{i(\nu)\}_{\nu=0}^\infty$ is called a \texttt{cyclic control sequence}, if $i(\nu)=\left(\nu \operatorname{mod} n\right)+1 $, where $n$ is the number of sets in \eqref{eq:DefCFP}.
\end{definition}

The \textit{cyclic subgradient projection} method can be written as algorithmic operator $\boldsymbol{T}$ in the following way:

Let $x^0\in\mathbb{R}^n$ be an arbitrary starting point. Given the current iterate $x^{k}$, the next iterate $x^{k+1}$ can be calculated via

\begin{align}\label{eq:CPMiterDef}
x^{k+1} &= \boldsymbol{T}(x^k) = x^{k} + \lambda _{k} p(x^k)
\end{align}
with

\begin{align}\label{eq:DefCPMStep}
p(x^k)= &-\frac{\max\{0, \varphi_{i(k)}(x^{k})\}}{\left\| \xi^{k}\right\| ^{2}
}\xi^{k}
\end{align}
where $\xi^{k}\in\partial \varphi_{i(k)}(x^{k})$ (subgradient of $\varphi_{i(k)}$ at $x^{k}$
) is arbitrary, $\lambda _{k}\in \lbrack \epsilon _{1},2-\epsilon _{2}]$ (relaxation parameters) for arbitrary $\epsilon _{1},\epsilon _{2}\in (0, 1]$ and $\{  i(k)\}$ is a cyclic control sequence.

\medskip

The second option we present for the algorithmic operator $\boldsymbol{T}$ is the \textit{simultaneous subgradient projection} method. We can write it in terms of $\boldsymbol{T}$ as follows:

Let $x^0\in\mathbb{R}^n$ be an arbitrary starting point. Given the current iterate $x^{k}$, the next iterate $x^{k+1}$ can be calculated via

\begin{align}\label{eq:SPMiterDef}
x^{k+1} &= \boldsymbol{T}(x^k) = x^{k} + \lambda _{k} p(x^k)
\end{align}
with

\begin{align}\label{eq:DefSPMStep}
p(x^k) &\colonequals -\sum_{i\in I}w_{i}\frac{\max\{0, \varphi_{i}(x^{k})\}}{\left\| \xi ^{k}\right\| ^{2}}\xi ^{k}
\end{align}
where $\xi ^{k}\in \partial \varphi_{i}(x^{k})$ (subgradient of $\varphi_{i}$ at $x^{k}$) is arbitrary, $w_{i}>0$ are weights with $\sum_{i\in I}w_i=1$ and $\lambda _{k}\in \lbrack \epsilon _{1},2-\epsilon _{2}]$ (relaxation parameters) for arbitrary $\epsilon _{1},\epsilon _{2}\in(0, 1]$.

\medskip

In the following, we use the notation $\bar{p}(x^k) \colonequals p(x^k)/\|p(x^k)\|$, refer to the simultaneous subgradient projection method simply as simultaneous projection and to the cyclic subgradient projection method as cyclic projection.

The convergence proof of the level set scheme \cite[Theorem 3.6]{gkrs18} relies on using a finite convergence method. We can transform both the cyclic and the simultaneous projection to fulfill this property using the approach of \cite{pi1988}. Note that \cite[Theorem 3.6]{gkrs18}, in contrast to convergence proofs for other finite convergent projection methods, e.g. \cite{ccp11}, is not based on the assumption that the Slater Condition holds. In our context this means that we do not rely on the existence of $x\in\mathbb{R}^n$ with $\varphi_{i}(x) < 0$ for all $i\in I$ for the level set scheme to converge.
\medskip

Perturbations can be applied to any algorithmic operator $\boldsymbol{T}$ which generates a sequence $\{x^k\}_{k=0}^{\infty}$ via $x^{k+1} = \boldsymbol{T}(x^k)$.
A central concept in the context of perturbations is \textit{bounded perturbation resilience}. This concept is defined in \cite{bdhk07} in a general way, i.e. for a general mathematical problem $\boldsymbol{P}$ and an algorithmic operator $\boldsymbol{T}$, which is suited to solve $\boldsymbol{P}$. \\

In \cite{dgj18,dgjt17} the concept of bounded perturbation resilience is considered with the problem $\boldsymbol{P}$ being the variational inequality problem. The operators $\boldsymbol{T}$ used to solve $\boldsymbol{P}$ in \cite{dgj18,dgjt17} are the extragradient method, the subgradient extragradient method and the projection and contraction method. In \cite{gh14,jcj13} the problem is the bioluminescence imaging problem, which can be phrased as a constrained optimization problem, and the operator used to solve it is the expectation maximization method. For a detailed overview and more examples see \cite{censor18}.

In this paper, the problem $\boldsymbol{P}$ is a CFP and $\boldsymbol{T}$ can be read as one of the projection methods we presented before.

\begin{definition}\label{eq:DefBPR}
	Given a problem $\boldsymbol{P}$, an algorithmic operator $\boldsymbol{T}$ and a starting point $x^0$ such that the sequence $\{x^{k}\}_{k=0}^{\infty},$ generated by $x^{k+1}=\boldsymbol{T}(x^{k})$ converges to a solution of $\boldsymbol{P}$. Then $\boldsymbol{T}$ is called \textit{bounded perturbation resilient} if any sequence $\{y^{k}\}_{k=0}^{\infty}$ with $y^0 = x^0$ generated using either \textit{inner perturbations} via
	
	\begin{align}
	y^{k+1} &= \boldsymbol{T}(y^k + \beta_k v^k)\qquad \forall k\geq 0
	\end{align}
	 or  using \textit{outer perturbations} via
	
	 \begin{align}
	 y^{k+1} &= \boldsymbol{T}(y^k) + \beta_k v^k\qquad \forall k\geq 0
	 \end{align}
	   where $\beta_{k}v^{k}$ are bounded perturbations (i.e. $\beta_k\in\mathbb{R}_{\geq 0}$ for all  $k\geq 0$, ${\displaystyle\sum\limits_{k=0}^{\infty}}\beta_{k}\,<\infty$, $v^k\in\mathbb{R}^n$ and $\|v^k\|\leq M\in\mathbb{R}$ for all $k\geq 0$) also converges to a solution of $\boldsymbol{P}$.
\end{definition}

Both the cyclic and the simultaneous projection method are known to be bounded perturbation resilient \cite{special-issue}.

\section{Perturbation of the simultaneous and cyclic subgradient projection method\label{sec:Pert}}

In this section we first present results on the convergence of sequences, which were generated using inner or outer perturbations. In Subsections \ref{sec:HBPerturbation} and \ref{sec:SCPerturbation} we introduce two specific perturbations developed by us, show how they can be formulated both as inner and outer perturbation and that they are bounded in the sense of Definition \ref{eq:DefBPR}. Finally, we summarize how we combined the tools presented in this section and Section \ref{sec:Prelim} in our algorithm.

\subsection{The relation of inner and outer perturbations}

In this section we explore the relation of iteration sequences resulting from methods using inner and outer perturbations concerning their convergence behavior.

First, we are extending Proposition 5 from \cite{CensorReem2014} to include the simultaneous subgradient projection method \eqref{eq:SPMiterDef} as algorithmic operator $\boldsymbol{T}$.

Let

\begin{align}
\tilde{b}^k &\colonequals \begin{cases}b^k & \text{if }c(x^k) = \text{ true}\\ 0 & \text{otherwise} \end{cases}
\end{align}
be perturbation vectors where $c:\mathbb{R}^n \rightarrow \{$false, true$\}$ is a function that gives a condition, which determines whether or not to perturb the current iterate.

Consider the sequences $\{y^k\}_{k=0}^{\infty}$ using outer perturbations and $\{z^k\}_{k=0}^{\infty}$ using inner perturbations, which are defined as follows:

\begin{align}\label{eq:DefOuterPert}
y^{k+1} & = \boldsymbol{T}(y^k) + \tilde{b}^k
\end{align}
and

\begin{align}\label{eq:DefInnerPert}
\begin{cases}
z^0 = \boldsymbol{T}(x^0)\\
z^{k+1} = \boldsymbol{T}(z^k + \tilde{b}^k)\end{cases}
\end{align}

\begin{proposition}\label{thm:prop5}
	Suppose that $\{b^k\}_{k=0}^{\infty}$ is a sequence in $\mathbb{R}^n$  satisfying $\lim_{k\rightarrow\infty} b^k \rightarrow 0$. If $\{y^k\}_{k=0}^{\infty}$ as defined in \eqref{eq:DefOuterPert} converges weakly to some $y^{\ast}$, then also $\{z^k\}_{k=0}^{\infty}$ as defined in \eqref{eq:DefInnerPert} converges weakly to $y^{\ast}$ and vice versa. If $\{y^k\}_{k=0}^{\infty}$ converges strongly, then $\{z^k\}_{k=0}^{\infty}$ converges strongly to the same limit and vice versa.
\end{proposition}

\begin{proof}
	We show that by induction $y^{k+1} = z^k + \tilde{b}^k$ for all $k\in\mathbb{N}\cup \{0\}$. The rest of the statement follows from the proof of Proposition 5 in \cite{CensorReem2014}.
	
	\paragraph{k=0}
	Show that $y^1 = z^0 + \tilde{b}^0$.
	\begin{align*}
	z^0 + \tilde{b}^0 &= \boldsymbol{T}(y^0) + \tilde{b}^0 =  y^1& \text{according to } \eqref{eq:DefInnerPert} \text{ and }\eqref{eq:DefOuterPert}\\
	\end{align*}
	
	\paragraph{Induction step}
	Show that $y^k = z^{k-1} + \tilde{b}^{k-1} \Rightarrow y^{k+1} = z^k + \tilde{b}^k$.
	
	\begin{align*}
	z^k + \tilde{b}^k &= \boldsymbol{T}(z^{k-1} + \tilde{b}^{k-1}) + \tilde{b}^k & \text{according to } \eqref{eq:DefInnerPert}\\
	&= \boldsymbol{T}(y^k)+ \tilde{b}^k = y^{k+1}& \text{according to } \eqref{eq:DefOuterPert}
	\end{align*}
\end{proof}

Proposition \ref{thm:prop5} implies that inner and outer perturbations of the simultaneous projection method using the same perturbation vectors $\tilde{b}^k$ as described before can be used interchangeably with regard to weak and strong convergence.

In Sections \ref{sec:HBPerturbation} and \ref{sec:SCPerturbation} we present two perturbations developed by us and formulate them both as inner and outer perturbation. In contrast to \eqref{eq:DefOuterPert} and \eqref{eq:DefInnerPert}, we do not use the same perturbation vectors $\tilde{b}^k$ for both formulations there.
Instead, we formulate them according to the following schemes. We assume specific choices of $\lambda_k, \beta_k$, which will be explained in more detail, and present a result regarding the convergence of the sequence of iterates produced by methods using these perturbations.

Let again $\boldsymbol{T}$ be an algorithmic operator with

\begin{align}\label{eq:GeneralDefOurPerturbedOperator}
\boldsymbol{T}(x) = x + \lambda(x) p(x)
\end{align}
where $p(x)$ is either the simultaneous \eqref{eq:DefSPMStep} or the cyclic projection step \eqref{eq:DefCPMStep}. Note that $\lambda:\mathbb{R}^n\rightarrow [\epsilon_{1}, 2-\epsilon_{2})$ with arbitrary $\epsilon_{1},\epsilon_{2}\in(0,1] $ is a function here.

Let $c:\mathbb{R}^n\rightarrow \{$false, true$\}$ and $\tilde{c}:\mathbb{R}^n \times \mathbb{R}^n \rightarrow \{$false, true$\}$ with

\begin{align}\label{eq:DefTildePertCondition}
\tilde{c}(y^k, y^{k-1}) &\colonequals \neg c(y^{k-1}) \wedge c(y^k).
\end{align}
be functions, which determine, whether perturbation is applied in a certain iteration. $\tilde{c}(y^k, y^{k-1})$ is by construction only true, if $ c(y^{k-1})$ is false and $c(y^k)$ is true.

Let $\{y^k\}_{k=0}^{\infty}$ be the sequence of iterates produced by the outer perturbation scheme, which is defined as follows.

\begin{align}\label{eq:OuterPerturbationScheme}
y^{k+1} &= \boldsymbol{T}(y^k) + \beta_{k} v^k
\end{align}
where $\{\beta_k\}_{k=0}^{\infty}$ is a sequence with $\beta_k\in\mathbb{R}_{\geq 0}$ for all  $k\geq 0$ and ${\displaystyle\sum\limits_{k=0}^{\infty}}\beta_{k}\,<\infty$,

\begin{align}\label{eq:DefOuterPertVectors}
v^k &= \begin{cases} b^k - \lambda(y^k) p(y^k) &\text{if }\tilde{c}(y^k, y^{k-1})= \text{ true}\\ 0 & \text{otherwise}
\end{cases}
\end{align}
and the sequence $\{b^k\}_{k=0}^{\infty}$ is bounded. Furthermore assume that $\|p(y^k)\| < q \in\mathbb{R}$ for all $k\geq 0$.

Let $\{z^k\}_{k=0}^{\infty}$ be the sequence of iterates produced by the inner perturbation scheme, which is defined as follows.

\begin{align}\label{eq:InnerPerturbationScheme}
z^{k+1} &= \boldsymbol{T}(z^k + \beta_{k} v^k)
\end{align}

where $\{\beta_k\}_{k=0}^{\infty}$ is a sequence with $\beta_k\in\mathbb{R}_{\geq 0}$ for all  $k\geq 0$ and ${\displaystyle\sum\limits_{k=0}^{\infty}}\beta_{k}\,<\infty$,

\begin{align}\label{eq:DefInnerPertVectors}
v^k &= \begin{cases} b^k &\text{if }c(z^k) = \text{ true}\\ 0 & \text{otherwise}
\end{cases}
\end{align}
and the sequence $\{b^k\}_{k=0}^{\infty}$ is bounded.

Now we choose

\begin{align}\label{eq:BinaryBetas}
\beta_k &= \begin{cases}1 & k\leq K\\ 0 & k> K\end{cases}
\end{align}
for some $K\in\mathbb{N}$. Choosing $\beta_k$ like this means that for $k>K$ no more perturbations will be applied. In our computations we choose $K$ as the maximum number of iterations allowed for the projection method to solve the current CFP.

\begin{lemma}\label{lemma:sequenceComparison}
	If $y^0 = z^0$ and $\{\beta_k\}_{k=0}^{\infty}$ is chosen as in \eqref{eq:BinaryBetas}, the following statements are true.
	\begin{itemize}
		\item[(a)] If  for some $l<K$ it holds that $c(z^k) = $ false for all $k = 0,..., l-1$  and  $c(z^{l}) = $ true, we have $z^k = y^k$ for all $k = 0,..., l$.
		\item[(b)] $\{z^k\}_{k=0}^{\infty}\subseteq \{y^k\}_{k=0}^{\infty}$.
		\item[(c)] For all $z^k\in \{z^k\}_{k=0}^{\infty}$ exists $N\in\mathbb{N}_0,  0\leq N\leq K$ such that $z^k = y^{k+N}$.
	\end{itemize}
\end{lemma}

\begin{proof}
	Statement (a) follows directly from equations \eqref{eq:DefInnerPertVectors} and \eqref{eq:DefOuterPertVectors}.
	
	Now suppose, that $l$ with $0\leq l < K$ is the first iteration index where the condition for applying perturbations is fulfilled. This means that $c(z^{l}) = \tilde{c}(y^{l}, y^{l-1}) = $ true  and $c(z^k) = \tilde{c}(y^k, y^{k-1}) = $ false for all $k<l$. We know from (a) that then $z^k = y^k$ for all $k = 0,..., l$.
	
	We have $\beta_l = 1$ because $l<K$. The iteration schemes give us
	
	\begin{align*}
	z^{l+1} &= \boldsymbol{T}(z^l + \beta_l b^l) = \boldsymbol{T}(y^l + b^l)
	\end{align*}	
	and
	
	\begin{align*}
	y^{l+1} &= \boldsymbol{T}(y^l) + \beta_l (b^l - \lambda(y^l) p(y^l))\\
	&=  y^l + (1-\beta_l) \lambda(y^l) p(y^l) + \beta_l b^l = y^l + b^l.
	\end{align*}
	
	$c(y^l)=$ true, so $\tilde{c}(y^{l+1}, y^l) = $ false, no matter what $c(y^{l+1})$ is. The following iterate is therefore an unperturbed one.
	
	\begin{align*}
	y^{l+2} &= \boldsymbol{T}(y^{l+1}) =  \boldsymbol{T}(y^l + b^l) = z^{l+1}
	\end{align*}
	
	Because $\lambda$ depends on the iterate (and not on the iteration index) and $y^{l+2} = z^{l+1}$, we get
	
	\begin{align*}
	z^{l+2} &= \boldsymbol{T}(z^{l+1}) = z^{l+1} + \lambda(z^{l+1}) p(z^{l+1})\\
	&= y^{l+2} + \lambda(y^{l+2}) p(y^{l+2}) = \boldsymbol{T}(y^{l+2}) = y^{l+3}
	\end{align*}	
	assuming that $c(z^{l+1})=$ false. Otherwise, the same argument as before holds.
	Statement (b) follows from this. Before any perturbations are applied, the number $N$ from statement (c) is 0. Each time perturbations are applied in an iteration, $N$  increases by 1.
	
\end{proof}

\begin{proposition}\label{thm:SequenceConvergenceComparison}
	Suppose that $\{\beta_k\}_{k=0}^{\infty}$ is chosen as in \eqref{eq:BinaryBetas} and $\{b^k\}_{k=0}^{\infty}$ is a bounded sequence in $\mathbb{R}^n$. If $\{y^k\}_{k=0}^{\infty}$ converges weakly to some $y^{\ast}$, then also $\{z^k\}_{k=0}^{\infty}$ converges weakly to $y^{\ast}$ and vice versa. If $\{y^k\}_{k=0}^{\infty}$ converges strongly, then $\{z^k\}_{k=0}^{\infty}$ converges strongly to the same limit and vice versa.
\end{proposition}

\begin{proof}
	It follows from Lemma \ref{lemma:sequenceComparison}(c) that there exists $0\leq N\leq K$ such that $z^k = y^{k+N}$ for all $k=l,... \infty$ for some $l\in\mathbb{N}_0$.
\end{proof}

Proposition \ref{thm:SequenceConvergenceComparison} implies that inner and outer perturbations of both the simultaneous and cyclic projection method using the perturbation vectors \eqref{eq:DefInnerPertVectors} or \eqref{eq:DefOuterPertVectors} can be used interchangeably with regard to weak and strong convergence.

\subsection{The heavy ball perturbation}\label{sec:HBPerturbation}

An early example of inertial-type methods (also known as $k$-step methods) is the \textit{heavy ball} method of \cite{Polyak:64}. Inertial-type methods are a time-discretization of an ordinary differential equation defining a continuous-time dynamical system (in general are easier to understand than their discrete-time counterparts, see \cite{Borwein:2018aa}). These methods incorporate gradient information from the last $k$ iterates into the calculation of the iteration step towards the next iterate and can be used to avoid zigzagging behavior. It is shown that methods using such inertial terms progress converge faster than methods using only current gradient information.
For a deeper discussion of this matter, see e.g. \cite{acpr:18,Nesterov:83} as well as \cite{dgj18,dgjt17} and the many references therein.

Our approach is to use such terms as perturbations for algorithmic operators $\boldsymbol{T}$ as defined in \eqref{eq:GeneralDefOurPerturbedOperator} and show that such interference can accelerate the convergence of the algorithm.

\subsubsection{Formulation as inner perturbation}

The iteration scheme using the heavy ball perturbation as inner perturbation is the following:

\begin{align}
x^{k+1} &= \boldsymbol{T}(x^k + \beta_{k} v^k)
\end{align}
where

\begin{align}\label{eq:HBinnerPerturbationDef}
v^k &= \begin{cases}
\lambda^{HB}_k d^{HB} & \text{if }c(x^k) = \text{ true}\\
0 & \text{otherwise}
\end{cases}
\end{align}

$\{\beta_{k}\}$ is a sequence with $\beta_k\in\mathbb{R}_{\geq 0}$ for all $k\geq 0$ and $\sum_{k=1}^{\infty} \beta_k < \infty$, $\{\lambda^{HB}_k\}$ is a bounded user-chosen sequence of step lengths and $d^{HB} = \bar{p}(x^{k-1}) + \bar{p}(x^k)$, where $p(x)$ is either \eqref{eq:DefCPMStep} or \eqref{eq:DefSPMStep}. The perturbation direction $d^{HB}$ is chosen like this, similar to iteration rules of $k$-step methods, in order to make opposing parts of $\bar{p}(x^{k-1})$ and $\bar{p}(x^k)$ cancel each other, because those parts are provoking the zigzagging behavior. Instead we get a direction which contains what both vectors have in common.
The function $c$, which determines, whether perturbations are applied or not, is defined as follows.

\begin{align}\label{eq:DefPertCondition}
c(x^k) = \begin{cases}
\text{true } & \text{if }  \langle\bar{p}(x^{k-1}), \bar{p}(x^k)\rangle \in [-1 + \epsilon_{\min}, -1+\epsilon_{\max}]\\
\text{false } & \text{otherwise }
\end{cases}
\end{align}

Choosing $c$ like this means, that perturbations are applied, when $\langle\bar{p}(x^{k-1}), \bar{p}(x^k)\rangle$ is close to $-1$, which translates to $\bar{p}(x^{k-1})$ and $\bar{p}(x^k)$ pointing into almost opposite directions and therefore zigzagging behavior. If that is the case, the convergence of simultaneous and cyclic projection is slow and our aim is to accelerate it by using perturbations.

Recalling Definition \ref{eq:DefBPR}, we now show that the perturbations \eqref{eq:HBinnerPerturbationDef} are bounded.

\begin{lemma}
	Let  $c(x^k) =$ true. Choose a bounded sequence $\{\lambda^{HB}_k\}$ with $\lambda^{HB}_k \geq 0$ for all $k\geq 0$. Then, the perturbations $\beta_{k} v^k$ defined in \eqref{eq:HBinnerPerturbationDef} are bounded.
\end{lemma}

\begin{proof}
	
	With $\{\beta_k\}, \{\lambda^{HB}_k\}$ chosen as described in the assumptions, it suffices to show that $\|d^{HB}\|$ is bounded.
	
	\begin{align*}
	\|d^{HB}\| &= \|\bar{p}(x^{k-1}) + \bar{p}(x^k)\|\\
	&= \sqrt{\|\bar{p}(x^{k-1})\|^2 + \|\bar{p}(x^k)\|^2 + 2\langle\bar{p}(x^{k-1}), \bar{p}(x^k)\rangle}\\
	&= \sqrt{2 + 2 (-1 + \epsilon)} \leq \sqrt{2 \epsilon_{\max}} =: M
	\end{align*}
	
\end{proof}

\subsubsection{Formulation as outer perturbation}

The iteration scheme using the heavy ball perturbation as outer perturbation is the following:

\begin{align}
x^{k+1} &= \boldsymbol{T}(x^k) + \beta_{k}v^k
\end{align}
where

\begin{align}\label{eq:HBouterPerturbationDef}
v^k &= \begin{cases}
\lambda^{HB}_k d^{HB} -  \lambda(x^k) p(x^k) & \text{if } \tilde{c}(x^k, x^{k-1}) = \text{true}\\
0 & \text{otherwise}
\end{cases}
\end{align}

$\{\beta_{k}\}$ is a sequence with $\beta_k\in\mathbb{R}_{\geq 0}$ for all $k\geq 0$ and $\sum_{k=1}^{\infty} \beta_k < \infty$, $\{\lambda^{HB}_k\}$ is a user-chosen bounded sequence of step lengths and $d^{HB} = \bar{p}(x^{k-1}) + \bar{p}(x^k)$, where $p(x)$ is either \eqref{eq:DefCPMStep} or \eqref{eq:DefSPMStep}. The function $\tilde{c}$ is defined as in \eqref{eq:DefTildePertCondition} and \eqref{eq:DefPertCondition}.

Recalling Definition \ref{eq:DefBPR}, we now show that the perturbations \eqref{eq:HBouterPerturbationDef} are bounded.

\begin{lemma}
	Assume, that  $\|p(x^k)\| \leq q \in\mathbb{R}$ for all $k\geq 0$ and $\{\lambda^{HB}_k\}$ is a bounded sequence which fulfills $\lambda^{HB}_k \geq 0$ for all $k\geq 0$.
	
	Then, the perturbations $\beta_{k} v^k$ defined in \eqref{eq:HBouterPerturbationDef} are bounded.
\end{lemma}

\begin{proof}
	With $\{\beta_k\}$ chosen as described in the assumptions, it suffices to show that $\|v^k\| < \bar{M}\in\mathbb{R}$ for all $k\geq 0$. We have shown before, that $\|d^{HB}\|< M\in \mathbb{R}$.
	
	\begin{align*}
	\|v^k\|^2 &\leq \|\lambda^{HB}_k d^{HB} - \lambda(x^k) p(x^k)\|^2\\
	&= 	\|\lambda^{HB}_k d^{HB}\|^2 + \|\lambda(x^k) p(x^k)\|^2 - 2\langle \lambda^{HB}_k d^{HB}, \lambda(x^k) p(x^k)\rangle\\
	&\leq (\lambda^{HB}_k)^2 M^2+ \lambda(x^k)^2 q^2 - 2 \lambda^{HB}_k \lambda(x^k) \left(\langle \bar{p}(x^k), p(x^k)\rangle + \langle \bar{p}(x^{k-1}), p(x^k)\rangle\right)\\
	&=   (\lambda^{HB}_k)^2 M^2+ \lambda(x^k)^2 q^2 - 2 \lambda^{HB}_k \lambda(x^k) \left(\|p(x^k)\|(1 + -1 + \epsilon)\right)\\
	&=  (\lambda^{HB}_k)^2 M^2 + \lambda(x^k)^2 q^2 - \underbrace{2 \lambda^{HB}_k \lambda(x^k) \|p(x^k)\| \epsilon}_{>0}\\
	&<  (\lambda^{HB}_k)^2 M^2 +4 q^2  =: \bar{M}
	\end{align*}
\end{proof}

\subsection{The surrogate constraint perturbation}\label{sec:SCPerturbation}

The approach we used for the surrogate constraint perturbation originates from a special kind of \textit{surrogate constraint} method presented in  \cite{Dudek2007}. In this method, the author combines the traditional idea of surrogate constraint methods (see e.g. \cite{Yang1992}) with relaxed, averaged projections. This algorithm has the ability to prevent zigzagging behavior in solving a linear feasibility problem.

Again, we consider an algorithmic operator $\boldsymbol{T}$ as defined in \eqref{eq:GeneralDefOurPerturbedOperator}.

We modified the idea of \cite{Dudek2007} in such a manner that we did not project the individual gradients of the violated constraint functions onto the surrogate half-space $H^{SC}$. Instead, we calculated $p(x^k)$ as defined in \eqref{eq:SPMiterDef} and then projected it onto the half-space $H^{SC}$.

In Figure \ref{fig:SCMPerturbation} the vectors $\bar{p}(x^{k-1})$ and $p(x^{k})$ are depicted in red. We define the half-spaces

\begin{align}
H^{SC}\colonequals H^{k-1} &:= \{x\in\mathbb{R}^n :  \langle x-x^{k}, p(x^{k-1}) \rangle \geq 0\}\label{eq:DefSurrogateHalfspaceDudek}\\
H^{k} &:= \{x\in\mathbb{R}^n :  \langle x-(x^k+p(x^k)), p(x^{k}) \rangle \geq 0\}.
\end{align}

The vector 

\begin{align}\label{eq:DefSCpertDirection}
d^{SC} &\colonequals P_{H^{SC}}(x^k + p(x^k)) - x^{k}
\end{align}
is the vector pointing from $x^k$ to the metric projection of the point $x^k + p(x^k)$ onto the surrogate half-space $H^{SC}$. It is depicted in cyan in Figure \ref{fig:SCMPerturbation}. Projecting onto $H^{SC}$ eliminates the part of $p(x^k)$ which points into the direction of $-p(x^{k-1})$. This means that $d^{SC}\perp p(x^{k-1})$ and steering the iteration into the direction of $d^{SC}$ reduces zigzagging behavior.

By $A$ we denote the intersection point of $H^{k-1}$, $H^{k}$ and $span\{p(x^{k-1}), p(x^k))\}$. Let $\alpha := \cos^{-1}(1-\epsilon)$ be the angle between $-p(x^{k-1})$ and $p(x^{k})$ and $\beta \colonequals \pi/2 - \alpha$.
Let $B\colonequals P_{H^{SC}}(x^k + p(x^k))$ and  $C\colonequals x^k + p(x^k)$ and . The triangles $\Delta C B x^k  $ and  $\Delta C A x^{k}$ are similar, so it is true that

\begin{figure}
\center
	\includegraphics[height=9cm]{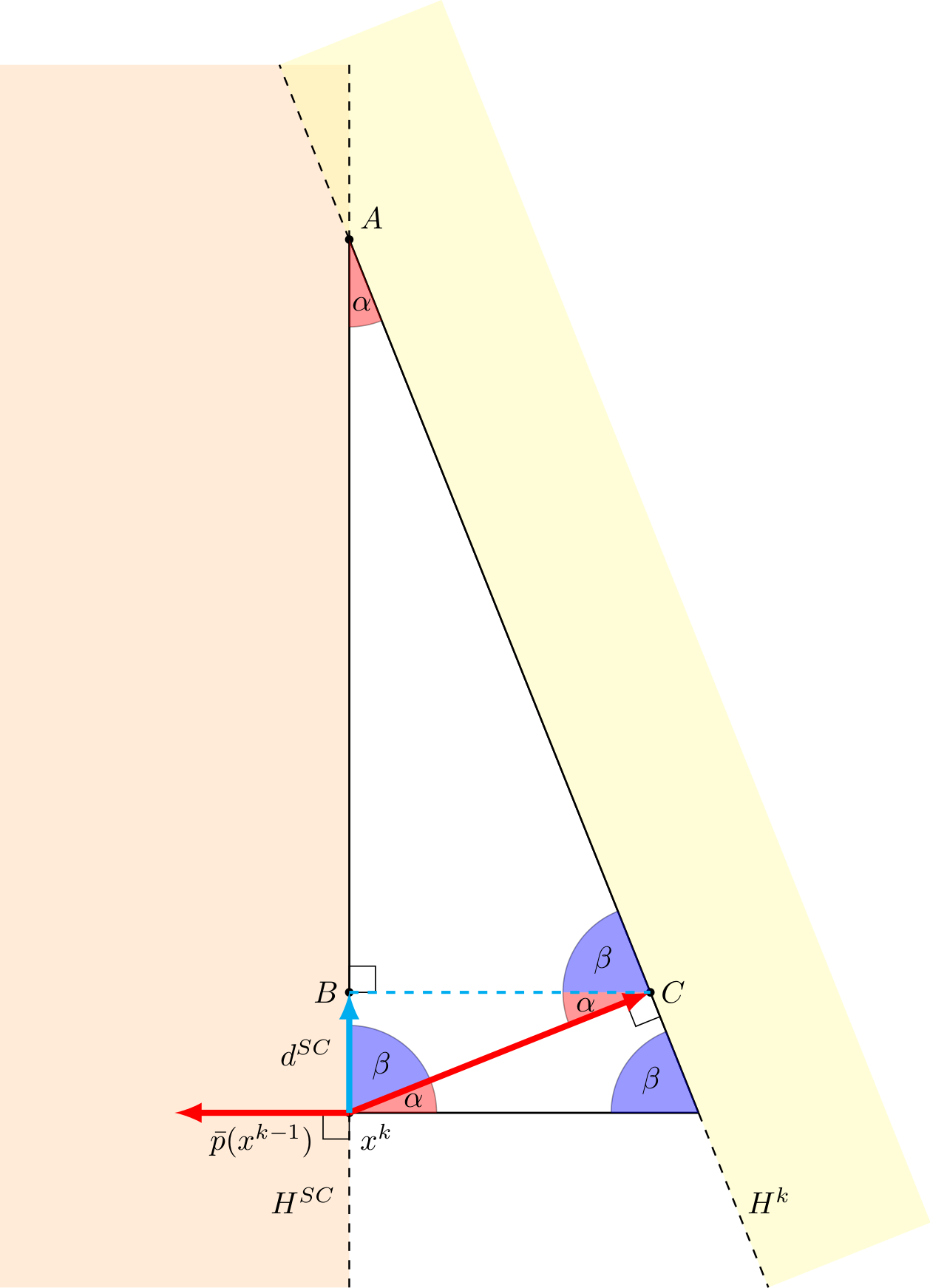}
	\caption{The direction $d^{SC}$ (depicted in cyan) of the surrogate constraint perturbation is calculated by projecting the point $x^k+p(x^{k})$ onto the surrogate half-space $H^{SC}$.}
	\label{fig:SCMPerturbation}
\end{figure}

\begin{equation*}
\frac{\|A - x^k\|}{\|p(x^{k})\|} = \frac{\|p(x^{k})\|}{\|d^{SC}\|}.
\end{equation*}

Choosing the step length

\begin{align}\label{eq:DefSCpertScalingFactor}
\lambda^{SC}_k = \|p(x^{k})\|^2 / \|d^{SC}\|^2
\end{align}
we get

\begin{equation*}
\frac{\|p(x^{k})\|}{\|d^{SC}\|} = \frac{\lambda^{SC}_k\|d^{SC}\|}{\|p(x^{k})\|}.
\end{equation*}

 These equations yield $\|A-x^k\| = \|\lambda^{SC}_k d^{SC}\|$. The construction of $d^{SC}$ implies that $d^{SC}/\|d^{SC}\| = (A-x^k)/\|A-x^k\|$ and therefore $A - x^k = \lambda^{SC}_k d^{SC}$. This means that the point $x^k + \lambda^{SC}_k d^{SC}$ is actually the same as the point $A$ depicted in Figure \ref{fig:SCMPerturbation}. Note that

\begin{align}
\sin(\alpha) &= \|d^{SC}\| / \|p(x^k)\|\label{eq:SCsinAlpha}\\
\Rightarrow \lambda^{SC}_k &=1/\sin(\alpha)^2\nonumber\\
\Rightarrow \|\lambda^{SC}_k d^{SC}\| &= \|p(x^k)\|/\sin(\alpha).\label{eq:SCpertStepLength}
\end{align}

As before using the heavy ball perturbation, we perturb the iteration sequence generated by $\boldsymbol{T}$ if $\langle\bar{p}(x^k), \bar{p}(x^{k-1})\rangle = -1+\epsilon$ for $\epsilon \in [\epsilon_{\min}, \epsilon_{\max}] \subset (0, 1)$.

\subsubsection{Formulation as inner perturbation}

The iteration scheme using the surrogate constraint perturbation as inner perturbation is

\begin{align}
x^{k+1} &= \boldsymbol{T}(x^k + \beta_k v^k)
\end{align}
where $\{\beta_k\}$ is a sequence with $\beta_k\in\mathbb{R}_{\geq 0}$ for all  $k\geq 0$ and ${\displaystyle\sum\limits_{k=0}^{\infty}}\beta_{k}\,<\infty$  and

\begin{align}\label{eq:SCinnerPertDef}
v^k &= \begin{cases}
\lambda^{SC}_k d^{SC} & \text{if } c(x^k) = \text{ true}\\
0 & \text{otherwise}
\end{cases}
\end{align}

Again, $c$ is defined as in \eqref{eq:DefPertCondition}.

As before in Section \ref{sec:HBPerturbation}, we now show that the perturbations \eqref{eq:SCinnerPertDef} are bounded in the sense of the definition \ref{eq:DefBPR} of bounded perturbation resilience.

\begin{lemma}
	Let $\langle\bar{p}(x^{k}), \bar{p}(x^{k-1})\rangle = -1+\epsilon$ for $\epsilon \in [\epsilon_{\min}, \epsilon_{\max}] \subset (0, 1)$. Further assume that $\|p(x^{k})\| \leq q \in\mathbb{R}$ holds for all $k\geq 0$.
	
	Then, the perturbations $\beta_{k} v^k$ as defined in \eqref{eq:SCinnerPertDef} are bounded.
\end{lemma}

\begin{proof}
	With $\{\beta_k\}$ chosen as described in the assumptions, it suffices to show that $\|v^k\| \leq M\in\mathbb{R}$ for all $k\geq 0$. Recall that $\alpha = \cos^{-1}(1-\epsilon)$  as in Figure \ref{fig:SCMPerturbation}. We know from \eqref{eq:SCpertStepLength}
	\begin{align*}
	\|v^k\| \leq \|\lambda^{SC}d^{SC}\| &= \|p(x^k)\|/\sin(\alpha)\\
	&= \frac{\|p(x^k)\|}{\sqrt{2\epsilon - \epsilon^2}}\\
	&\leq  \frac{q}{\sqrt{2\epsilon_{\min} - \epsilon_{\min}^2}}  =: M
	\end{align*}
\end{proof}

\subsubsection{Formulation as outer perturbation}

The iteration scheme using the surrogate constraint perturbation as outer perturbation is

\begin{align}
x^{k+1} &= \boldsymbol{T}(x^{k}) + \beta_k v^k
\end{align}
where $\{\beta_k\}$ is a sequence with $\beta_k\in\mathbb{R}_{\geq 0}$ for all  $k\geq 0$ and ${\displaystyle \sum \limits_{k=0}^{\infty}}\beta_{k}\,<\infty$  and

\begin{align}\label{eq:SCouterPertDef}
v^k &= \begin{cases}
\lambda^{SC}_k d^{SC} -  \lambda(x^k) p(x^k) & \text{if } \tilde{c}(x^k, x^{k-1}) = \text{ true}\\
0 & \text{otherwise}
\end{cases}
\end{align}

The function $\tilde{c}$ is the same as in \eqref{eq:DefTildePertCondition} with $c$ as in \eqref{eq:DefPertCondition}.

As before in Section \ref{sec:HBPerturbation}, we now show that the perturbations \eqref{eq:SCouterPertDef} are bounded in the sense of the Definition \ref{eq:DefBPR}.

\begin{lemma}
	Let $\langle\bar{p}(x^{k-1}), \bar{p}(x^{k})\rangle = -1+\epsilon$ for $\epsilon \in [\epsilon_{\min}, \epsilon_{\max}] \subset (0, 1)$.
	Further assume that $\|p(x^{k})\| \leq q \in\mathbb{R}$ holds for all $k\geq 0$.
	
	Then, the perturbations $\beta_{k} v^k$ as defined in \eqref{eq:SCouterPertDef} are bounded.\\
\end{lemma}

\begin{proof}
	
	With $\{\beta_k\}$ chosen as described in the assumptions, it suffices to show that $\|v^k\| \leq \bar{M}\in\mathbb{R}$ for all $k\geq 0$. We have shown before that $\|\lambda^{SC}d^{SC}\| \leq M\in\mathbb{R}$.
	
	Let $\beta = \cos^{-1}(\langle\bar{p}(x^{k-1}), \bar{p}(x^{k})\rangle)$. We have $d^{SC}\in span \{p(x^{k-1}), p(x^{k})\}, d^{SC}\perp p(x^{k-1})$ and $1/\|d^{SC}\| \langle d^{SC}, \bar{p}(x^{k})\rangle>0$. Therefore, $\angle(d^{SC}, p(x^{k})) = \beta - \pi/2$, which lies in the open interval $(0, \pi/2)$.\\
	
	We also know that
	\begin{align*}
	\cos(\beta - \pi/2) &= \sin(\beta)\\
	&= \sqrt{1-\cos(\beta)^2}\\
	&= \sqrt{1-(-1+\epsilon)^2}\\
	&= \sqrt{2\epsilon - \epsilon^2}
	\end{align*}
	
	\begin{align*}
	\|v^k\|^2 &\leq \|\lambda^{SC}d^{SC} - \lambda(x^k) p(x^k)\|^2\\
	&= 	\|\lambda^{SC}d^{SC}\|^2 + \|\lambda(x^k) p(x^k)\|^2 - 2\langle\lambda^{SC}d^{SC}, \lambda(x^k) p(x^k)\rangle\\
	&\leq M^2 + \lambda(x^k)^2 q^2 - 2 \lambda(x^k) \lambda^{SC} \|p(x^k)\|  \|d^{SC}\|  \langle d^{SC}/\|d^{SC}\|, \bar{p}(x^k)\rangle \\
	&= M^2 + \lambda(x^k)^2 q^2 - \underbrace{2 \lambda(x^k) \lambda^{SC} \|p(x^k)\|  \|d^{SC}\| \sqrt{2\epsilon - \epsilon^2}}_{> 0}\\
	& <  M^2 + 4 q^2 =: \bar{M}
	\end{align*}
\end{proof}

\subsection{Convergence speed of the perturbed projection methods}

We have shown that the property of boundedness in the sense of Definition \ref{eq:DefBPR} is fulfilled by the heavy ball and the surrogate constraint perturbation, no matter whether they are formulated as inner or outer perturbation. Because both simultaneous and cyclic projection are bounded perturbation resilient, we can use the heavy ball and the surrogate constraint perturbation together with these methods and retain convergence to a solution of $\boldsymbol{P}^s$, if such a solution exists. \cite[Theorem 11 and Corollary 14]{brz17} states that we retain the convergence rate of the simultaneous or cyclic projection when we use them with bounded perturbations. In the Section \ref{sec:Num} we show numerical results, which indicate that in some cases using the heavy ball or the surrogate constraint perturbation can even speed up the convergence. In this section we will elaborate on theoretical arguments, which illustrate in which cases the perturbed simultaneous projection method converges faster than its unperturbed counterpart.\medskip

Statements about the convergence speed of both the simultaneous projection method and the method introduced by Dudek \cite{Dudek2007} exist in the literature. Cegielski \cite[Theorem 4.4.5]{cegielski13} states the following result about the convergence speed of the simultaneous projection method. It is phrased in terms of metric projections $P_{C_i}$ onto closed convex sets $C_i$.

\begin{theorem}
	If $z\in C\colonequals \bigcap_{i\in I} C_i, C \neq \emptyset, \lambda \in (0, 2]$ and
	
	\begin{align}
	x^{k+1} &= x^k + \lambda \sum_{i\in I} w_i(x^k) (P_{C_i}(x^k)-x^k)
	\end{align}	
	it holds that
	
	\begin{align}\label{eq:SPMConvergence_Cegielski}
	\|x^{k+1} - z\|^2&\leq \|x^k-z\|^2 - \lambda(2-\lambda) \sum w_i(x^k) \|P_{C_i}(x^k)-x^k\|^2.
	\end{align}
\end{theorem}

The essential property of the operators $P_{C_i}$ needed for the proof of this result is that they are cutters, i.e.

\begin{align}\label{eq:Def_Cutter}
\langle P_{C_i}(x)-x, z-x\rangle \geq \|P_{C_i}(x)-x\|^2 \quad \forall z\in C_i, x\in\mathbb{R}^n.
\end{align}

\cite[Corrolary 4.2.6]{cegielski13} states that not only metric projections but also subgradient projections are cutters.
Therefore \eqref{eq:SPMConvergence_Cegielski}  also holds for the subgradient projections we introduced in \eqref{eq:DefSPMStep}. To distinguish between metric and subgradient projections, we use the following notation:

\begin{align}\label{eq:SPMConvergence_DoseEvalFuncs}
\tilde{P}_{C_i}(x) - x &= - \frac{\max\{0, \varphi_{i}(x)\}}{\left\| \xi\right\| ^{2}}\xi
\end{align}

where $\xi$ is an arbitrary element of the subdifferential $\partial \varphi_{i}(x)$  of $\varphi_i$ at $x$.\medskip

Dudek in \cite{Dudek2007} states a result about the convergence of their method with a very similar structure to \eqref{eq:SPMConvergence_Cegielski}. Their method is phrased as an algorithm to solve a system of linear inequalities and the projections considered are metric projections. Nevertheless, the result also applies to subgradient projections, because the proof relies on the fact that the projections are cutters.

Dudeks result rephrased in terms of subgradient projection operators $\tilde{P}_{C_i}$ is the following:

\begin{theorem}\label{thm:DudeksConvergenceResForSubgradientProj}
	If $z\in C\colonequals \bigcap_{i\in I} C_i, , C \neq \emptyset, \lambda \in [0,2), H^{SC}$ defined as in \eqref{eq:DefSurrogateHalfspaceDudek} and
	
	\begin{align}
	x^{SC} &= x^k + \lambda \frac{\sum_{i\in I} w_i(x^k) \|s^{i}\|^2}{\left\|\sum_{i\in I} w_i(x^k) d^{i}\right\|^2}d
	\end{align}
	with
	\begin{align}
	s^{i}&=\tilde{P}_{C_i}(x^k)-x^k\\
	d^{i} &= P_{H^{SC}}(\tilde{P}_{C_i}(x^k)-x^k)\\
	d &= \sum_{i\in I} w_i(x^k) d^{i}
	\end{align}
	it holds that
	
	\begin{align}\label{eq:SCMConvergence_Dudek}
	\|x^{SC} - z\|^2&\leq \|x^k-z\|^2 - \lambda(2-\lambda) \sum w_i(x) \|s^{i}\|^2 \frac{\sum w_i(x) \left\|s^{i}\right\|^2}{\left\|\sum w_i(x) d^{i}\right\|^2}
	\end{align}
\end{theorem}

Comparing \eqref{eq:SPMConvergence_Cegielski} and \eqref{eq:SCMConvergence_Dudek} we observe that the bound on the reduction $\|x^{k+1} - z\|^2 - \|x^{k} - z\|^2$ provided by the simultaneous projection and the bound on the reduction $\|x^{SC} - z\|^2 - \|x^k-z\|^2$ provided by Dudek's method differs only by the factor

\begin{align}
\delta(x^k)&\colonequals\frac{\sum w_i(x^k) \|\tilde{P}_{C_i}(x^k)-x^k\|^2}{\|\sum w_i(x^k) P_{H^{SC}}(\tilde{P}_{C_i}(x^k)-x^k)\|^2}
\end{align}

Because $P_{H^{SC}}$ is a metric projection onto a half-space we know that for all $k\geq 0$

\begin{align}
\left\|d^{i}\right\|^2 + \left\|(x^k + s^{i}) - P_{H^{SC}}(x^k + s^{i})\right\|^2 = \left\|s^{i}\right\|^2.
\end{align}

It follows that

\begin{align}
\left\|\sum w_i(x^k) d^{i}\right\|^2&\leq \sum w_i(x^k) \left\|d^{i}\right\|^2 \leq \sum w_i(x^k) \left\|s^{i}\right\|^2
\end{align}
and thus $\delta(x^k)\geq 1$ for all $k\geq 0$. The immediate consequence is that Dudek's method reduces the distance to any $z\in C$ at least as much as the simultaneous projection.

\subsubsection{Acceleration of the convergence by perturbations}
We consider the iteration step $k\rightarrow k+1$, which is perturbed using outer perturbations, i.e. we assume that $\tilde{c}(x^k, x^{k-1}) =$ true with the function $\tilde{c}$ as defined in \eqref{eq:DefTildePertCondition}.

\paragraph{\textbf{Surrogate constraint perturbation}}

\begin{lemma}
	\label{lemma:SCpertAndDudekAreTheSame}
	If $\langle p(x^{k-1}), \tilde{P}_{C_i}(x^k) - x^k \rangle \leq 0$ for all $i\in I$ then
	\begin{align}
	\sum_{i\in I} w_i(x^k) d^{i} = P_{H^{SC}}(p(x^k))
	\end{align}
\end{lemma}

\begin{proof}
	$-\bar{p}(x^{k-1})$ is the normal vector $\nu$ to the surrogate half-space $H^{SC}$ (with $\nu$ pointing away from $H^{SC}$).
	
	Let $v = \nu \langle \nu, v\rangle + u$ and $\langle u, \nu \rangle = 0$. From the definition of the metric projection it follows that
	
	\begin{align*}
	P_{H_{SC}} (v) &= u \Leftrightarrow \langle \nu, v\rangle\geq 0.
	\end{align*}
	
	Now let
	
	\begin{align*}
	v^{i} &\colonequals \tilde{P}_{C_i}(x^k)-x^k = \nu \langle \nu, v^{i}\rangle + u^{i}.
	\end{align*}
	
	The comparison of
	\begin{align*}
	\sum w_i(x^k) d^{i} &= \sum w_i(x^k) P_{H^{SC}}(\tilde{P}_{C_i}(x^k)-x^k)\\
	& = \sum w_i(x^k) u^{i}
	\end{align*}
	and
	\begin{align*}
	p(x^k) &= \sum w_i(x^k) v^{i}\\
	&= \sum w_i(x^k) (\nu \langle \nu, v^{i}\rangle + u^{i})\\
	&= \sum \nu  \langle \nu, w_i(x^k) v^{i}\rangle + w_i(x^k) u^{i}\\
	\Rightarrow P_{H^{SC}} (p(x^k)) &= \sum w_i(x^k) u^{i}
	\end{align*}	
	proves Lemma \ref{lemma:SCpertAndDudekAreTheSame}.
\end{proof}

Lemma \ref{lemma:SCpertAndDudekAreTheSame} implies that if $\langle p(x^{k-1}), \tilde{P}_{C_i}(x^k) - x^k \rangle \leq 0$ is fulfilled for all $i\in I$, then the surrogate constraint perturbation (with $\beta_k  =1$) and Dudek's method (with $\lambda  = 1$) actually result in the same iterate $x^{SC} \colonequals x^k +\lambda^{SC}_k d^{SC}$. This means that the statement of Theorem \ref{thm:DudeksConvergenceResForSubgradientProj} also applies for the next iterate resulting from the surrogate constraint perturbation. Note that the assumption $\langle p(x^{k-1}), \tilde{P}_{C_i}(x^k) - x^k \rangle \leq 0$ for all $i\in I$ is a reasonable assumption if there are two main groups of conflicting goals, each iteration step fulfills one group of goals, and the groups alternate in each iteration step.

Furthermore we can quantify the factor $\delta(x^k)$ in dependence of $\langle \bar{p}(x^k), \bar{p}(x^{k-1})\rangle$ under this assumption, as we will demonstrate in the following theorem. Let $\alpha  \colonequals \pi - \cos^{-1}(\langle \bar{p}(x^k), \bar{p}(x^{k-1})\rangle)$.

\begin{theorem}
	\label{thm:MySCpertConvergenceResult}
	If $\langle p(x^{k-1}), \tilde{P}_{C_i}(x^k) - x^k \rangle \leq 0$ for all $i\in I$ and
	
	\begin{align}
	x^{SC} = x^k + \lambda^{SC}_k d^{SC}
	\end{align}
	
	with $\lambda^{SC}_k$ and $d^{SC}$ as defined in \eqref{eq:DefSCpertDirection} and \eqref{eq:DefSCpertScalingFactor} then
	\begin{align}
	\|x^{SC}-z\|^2 \leq  \|x^k-z\|^2 -\frac{1}{\sin(\alpha)^2} \sum w_i(x^k) \|\tilde{P}_{C_i}(x^k)-x^k\|^2
	\end{align}
\end{theorem}

\begin{proof}
From Theorem \ref{thm:DudeksConvergenceResForSubgradientProj} we know
	
\begin{align*}
\|x^{SC}-z\|^2 \leq  \|x^k-z\|^2 - \sum w_i(x^k) \|\tilde{P}_{C_i}(x^k)-x^k\|^2 \frac{\sum w_i \|\tilde{P}_{C_i}(x^k)-x^k\|^2}{\|\sum w_i(x^k) d^{i}\|^2}.
\end{align*}	
The inequality
\begin{align*}
-\sum w_i(x^k) \|\tilde{P}_{C_i}(x^k)-x^k\|^2 \leq -\|\sum w_i(x^k) (\tilde{P}_{C_i}(x^k)-x^k)\|^2 = -\|p(x^k)\|^2,
\end{align*} 	
yields	
\begin{align*}
\|x^{SC}-z\|^2 \leq  \|x^k-z\|^2 - \sum w_i(x^k) \|\tilde{P}_{C_i}(x^k)-x^k\|^2 \left(\frac{\|p(x^k)\|}{\|d^{SC}\|}\right)^2.
\end{align*}
The result now follows from \eqref{eq:SCsinAlpha}.
\end{proof}

\paragraph{\textbf{Heavy ball perturbation}}

By $x^{HB} \colonequals x^k + 0.5 \lambda^{HB}(\bar{p}(x^k)+\bar{p}(x^{k-1}))$ we denote the next iterate generated by the perturbed iteration scheme using heavy ball perturbation. Let $\langle \bar{p}(x^k), \bar{p}(x^{k-1})\rangle = -1+\epsilon$.\\

\begin{figure}
	\centering
	\includegraphics[height=0.4\textheight]{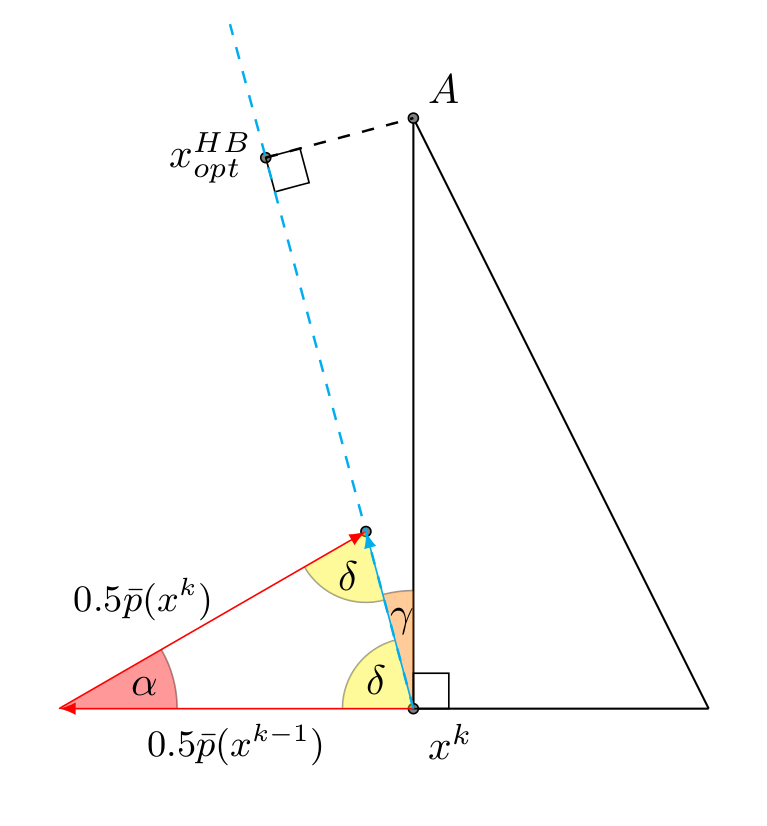}
	\caption{Illustration of the heavy ball perturbation}
	\label{fig:triangleHB}
\end{figure}

The triangle in Figure \ref{fig:triangleHB} is the same as in Figure \ref{fig:SCMPerturbation}. The Figure illustrates the progress towards the triangle tip $A$ made by the perturbed iteration scheme using heavy ball perturbation. The point $x^{HB}$ lies on the dashed cyan line, its exact position depending on the step length $\lambda^{HB}$.  Let $\gamma \colonequals\angle(A-x^k, x^{HB}-x^k)$. The triangle $x^k, x^k + 0.5\bar{p}(x^{k-1}), x^{HB}$ is isosceles, because $\|0.5\bar{p}(x^{k-1})\| = \|0.5\bar{p}(x^{k})\| = 0.5$. Therefore $\gamma = \pi/2 - (\pi-\alpha)/2 = \alpha/2$.

We have

\begin{align*}
\|x^{HB} - x^k\| &= 0.5\lambda^{HB} \|\bar{p}(x^k) + \bar{p}(x^{k-1})\|\\
&= 0.5\lambda^{HB} \sqrt{\|\bar{p}(x^k)\|^2 + \|\bar{p}(x^{k-1})\|^2 + 2\langle \bar{p}(x^k), \bar{p}(x^{k-1})\rangle}\\
&= \lambda^{HB} \sqrt{\frac{\epsilon}{2}}
\end{align*}

The distance $\|x^{HB} - A\|$ of course depends on the step size $\lambda^{HB}$. It attains its minimum if $\langle x^{HB} - A, x^{HB} - x^k\rangle = 0$, which occurs if

\begin{align}
\|x^{HB}_{opt} - x^k\| &= \cos(\alpha/2)\|x^{k}-A\|,
\end{align}
which means that

\begin{align}\label{eq:HBpertOptimalStepLength}
\lambda^{HB}_{opt} =  \cos(\alpha/2)\|x^{k}-A\| \sqrt{\frac{2}{\epsilon}}.
\end{align}

We get

\begin{align}\label{eq:HBpertOptimalDistToTriangleTip}
\|x^{HB}_{opt} - A\| &= \|x^k-A\|\sin(\alpha/2).
\end{align}

In Section \ref{sec:SCPerturbation} we have illustrated that $x^{SC}$ coincides with the triangle tip $A$.
Recall that $\alpha$ is tied to the parameters $\epsilon_{\min}, \epsilon_{\max}$ in \eqref{eq:DefPertCondition}, which are used to decide whether the perturbation is applied in a certain iteration $k$.

\begin{theorem}
	\label{thm:MyHBpertConvergenceResult}
	If $\langle p(x^{k-1}), \tilde{P}_{C_i}(x^k) - x^k \rangle \leq 0$ for all $i\in I$ then there exists $\tilde{\alpha}\in(0, \pi/3]$ such that if $\alpha = \pi-\cos^{-1}(\langle \bar{p}(x^k), \bar{p}(x^{k-1})\rangle) < \tilde{\alpha}$ is fulfilled and $\lambda^{HB}$ is chosen as in \eqref{eq:HBpertOptimalStepLength} the following statement holds:
	\begin{align}
	\|x^{HB}_{opt} -z\|^2 - \left(\|x^k-z\|^2 -\sum w_i \|\tilde{P}_{C_i}(x^k)-x^k\|^2 \right) < 0.
	\end{align}
\end{theorem}

\begin{proof}
	From \eqref{eq:HBpertOptimalDistToTriangleTip} it follows that for $x^{SC} = x^k + \lambda^{SC}_k d^{SC}$ and $x^{HB}_{opt} = x^k+0.5\lambda^{HB}_{opt} (\bar{p}(x^k) + \bar{p}(x^{k-1})) $ it is true that
	
	\begin{align*}
	\|x^{HB}_{opt} - x^{SC}\| &= \sin(\alpha/2) \|x^{SC}-x^k\|\\
	& = \frac{\sin(\alpha/2)\|p(x^k)\|}{\sin(\alpha)}.
	\end{align*}
	
	Theorem \ref{thm:MySCpertConvergenceResult} yields
	
	\begin{align*}
	\|x^{HB}_{opt} - z\|^2 &\leq \|x^{SC}-z\|^2 +\|x^{HB}_{opt} - x^{SC}\|^2 + 2 \|x^{SC}-z\| \|x^{HB}_{opt} - x^{SC}\|\\
	&\leq \|x^k-z\|^2 -\frac{1}{\sin(\alpha)^2} \sum w_i(x^k) \|\tilde{P}_{C_i}(x^k)-x^k\|^2\\
	&\qquad +\|x^{HB}_{opt} - x^{SC}\|^2 + 2 \|x^{SC}-z\| \|x^{HB}_{opt} - x^{SC}\|
	\end{align*}
	
	To show that the heavy ball perturbation reduces the distance to $z$ even more than the unperturbed simultaneous projection, we examine
	
	\begin{align*}
	&\|x^{HB}_{opt} - z\|^2 - \left(\|x^k-z\|^2 - \sum w_i(x^k) \|\tilde{P}_{C_i}(x^k)-x^k\|^2\right)\\
	\leq& \|x^{HB}_{opt} - x^{SC}\|^2 + 2 \|x^{SC}-z\|\|x^{HB}_{opt} - x^{SC}\|-\frac{\cos(\alpha)^2}{\sin(\alpha)^2} \sum w_i(x^k) \|\tilde{P}_{C_i}(x^k)-x^k\|^2
	\end{align*}
	
	\begin{align*}
	&\|x^{HB}_{opt} - x^{SC}\|^2 + 2 \|x^{SC}-z\|\|x^{HB}_{opt} - x^{SC}\|-\frac{\cos(\alpha)^2}{\sin(\alpha)^2} \sum w_i(x^k) \|\tilde{P}_{C_i}(x^k)-x^k\|^2\\
	=& \left(\|p(x^k)\|\frac{\sin(\alpha/2)}{\sin(\alpha)}\right)^2 + 2 \|x^{SC}-z\|\|p(x^k)\|\frac{\sin(\alpha/2)}{\sin(\alpha)} -\frac{\cos(\alpha)^2}{\sin(\alpha)^2} \sum w_i(x^k) \|\tilde{P}_{C_i}(x^k)-x^k\|^2\\
	\leq &\left(\|p(x^k)\|\frac{\sin(\alpha/2)}{\sin(\alpha)}\right)^2 + 2 \|x^{SC}-z\|\|p(x^k)\|\frac{\sin(\alpha/2)}{\sin(\alpha)} -\frac{\cos(\alpha)^2 \|p(x^k)\|^2}{\sin(\alpha)^2}\\
	=& \frac{\|p(x^k)\|}{\sin(\alpha)}\left(\frac{\|p(x^k)\|}{\sin(\alpha)}(\sin(\alpha/2)^2-\cos(\alpha)^2)+2 \|x^{SC}-z\| \sin(\alpha/2)\right)
	\end{align*}
	
	For $\alpha < \pi/3$ we have
	$$\sin(\alpha/2)^2-\cos(\alpha)^2< 0.$$
	
	For all (finite) $\|p(x^k)\|,\|x^{SC}-z\| $ there exists $\tilde{\alpha} \in (0, \pi/3]$ such that for all $\alpha\in[0, \tilde{\alpha})$:
	
	$$\left(\frac{\|p(x^k)\|}{\sin(\alpha)}(\sin(\alpha/2)^2-\cos(\alpha)^2)+2 \|x^{SC}-z\| \sin(\alpha/2)\right) < 0$$
	
	and thus $\|x^{HB}_{opt} - z\|^2 < \|x^k-z\|^2 - \sum w_i (x^k)\|\tilde{P}_{C_i}(x^k)-x^k\|^2$.
\end{proof}

Theorem \ref{thm:MyHBpertConvergenceResult} implies that under the mentioned assumptions there always exists a value for $\epsilon_{\max}$ with $\epsilon_{\max} >0$ as parameter for the function $\tilde{c}$, which guarantees that if $\langle \bar{p}(x^k), \bar{p}(x^{k-1}) \rangle < -1 + \epsilon_{\max}$ the reduction achieved by the heavy ball perturbation is larger than the reduction achieved by the unperturbed simultaneous projection.

\subsection{The algorithm}\label{sec:Algorithm}

Now we combine the tools we presented in the previous sections in our algorithm.

To solve a constrained convex optimization problem \eqref{eq:GeneralOptProblem} we translate it into its epigraph representation \eqref{eq:CFPepigraph}. Then we use the level set scheme to transform it into a sequence $\{\boldsymbol{P}^s\}$ of CFPs . Next, we use simultaneous \eqref{eq:SPMiterDef} or cyclic projection \eqref{eq:CPMiterDef} to solve each CFP $\boldsymbol{P}^s$. If the use of perturbations is specified, we alter the sequence of iterates generated by the projection methods using the heavy ball \eqref{eq:HBouterPerturbationDef} or the surrogate constraint perturbation \eqref{eq:SCouterPertDef}.

\section{Numerical demonstrations}\label{sec:Num}

In this section we present our numerical results. These were achieved with an algorithm implemented according to the description in Section \ref{sec:Algorithm}. It uses the heavy ball and the surrogate constraint perturbation with the outer perturbation scheme \eqref{eq:OuterPerturbationScheme}. As the results in this section will show, both perturbations are eminently useful to speed up the convergence of the algorithm towards a solution compared to the unperturbed methods.

\subsection{Linear feasibility problem}

We present a linear feasibility problem and demonstrate the behavior of both the simultaneous and cyclic projection method with and without perturbations when we use them to solve the problem stated below. We show that in this example, the unperturbed versions of the methods converge slower than the perturbed versions. Furthermore, we demonstrate that the control sequence for the cyclic projection method has an influence on whether the condition triggering the perturbations is fulfilled.

\subsubsection{Problem formulation}

We consider the system of linear inequalities

\begin{align}\label{eq:AcademicExample}
A x &\leq b
\end{align}
where

\begin{align}
A &= \begin{pmatrix}
-1/\delta_{x_1} & - 1/\delta_{x_2} & - 1/\delta_{x_3}\\
1/\delta_{x_1} & - 1/\delta_{x_2} & - 1/\delta_{x_3}\\
1/\delta_{x_1} & 1/\delta_{x_2} & - 1/\delta_{x_3} \\
-1/\delta_{x_1} & 1/\delta_{x_2} & - 1/\delta_{x_3}
\end{pmatrix}, b =
\begin{pmatrix}
-1\\-1\\-1\\-1
\end{pmatrix}
\end{align}
and $x = (x_1, x_2, x_3)^T\in\mathbb{R}^3$. In the following, by inequality $i$ we refer to the inequality $\langle a_i , x \rangle\leq b_i$, where $a_i$ denotes the $i$-th row of the matrix $A$ and $b_i$ is the $i$-th coordinate of $b$ for $i \in\{1,2, 3,4\}$.

These linear inequalities define half-spaces. The separating hyperplanes $H_i$ of these half-spaces intersect the $x_1$-axis at $\pm\delta_{x_1}$, the $x_2$-axis at $\pm\delta_{x_2}$ and the $x_3$-axis at $\delta_{x_3}$.

In this example, we choose $\delta_{x_3} = 100$, $\delta_{x_1} = \tan(\beta) \delta_{x_3}/\sin(\alpha)$ and $\delta_{x_2} = \tan(\beta) \delta_{x_3}/\cos(\alpha)$ with $\alpha = 30^{\circ}$ and $\beta = 5^{\circ}$.

Next we demonstrate the behavior of simultaneous and cyclic projection with and without perturbations when we use them to solve \eqref{eq:AcademicExample}.

\subsubsection{Results}

In what follows, we occasionally (especially in the descriptions and legends of the figures and tables) use abbreviations for the simultaneous projection (SP) and the cyclic projection method (CP) as well as for the heavy ball perturbation (HB) and the surrogate constraint perturbation (SC). CP+HB means for example that cyclic projection was used together with heavy ball perturbation.

For all methods, we choose the starting point $x^0 = (15, 0, 0)^{T}$ and the parameters $\epsilon_{max} = 6\cdot 10^{-2}, \epsilon_{min} = 10^{-6}$ and $\lambda^{SC}_k = \|p(x^{k})\|^2 / \|d^{SC}\|^2$ as described in Section \ref{sec:SCPerturbation}. We consider an iterate $x^{\ast}$ to be a solution of \eqref{eq:AcademicExample}, if $\|Ax^{\ast} - b\|_{\infty} \leq 10^{-10}$ and $K$ denotes the number of iterations needed to find $x^{\ast}$.

In contrast to the step size $\lambda^{SC}_k$ of the surrogate constraint perturbation, the choice of the step size $ \lambda^{HB}_k$ of the heavy ball perturbation is not motivated in a geometrical way. It can only be chosen in an optimal way, if the solution to the problem is already known. We present results for three different values of $\lambda^{HB}_k$ to demonstrate the effects of choosing it in an empirical way.

First, we use the simultaneous projection to solve \eqref{eq:AcademicExample}. We choose  $\lambda(x) \equiv 1.9$ to generate a sequence of iterates, which alternates between fulfilling inequalities $\{1, 4\}$ and $\{2, 3\}$. In this way, we obtain a sequence of projection steps $\{p(x^k)\}$ with $\langle \bar{p}(x^k), \bar{p}(x^{k-1}) \rangle \in [-1 + \epsilon_{min}, -1+\epsilon_{max}]$. If we use perturbations in this setting, they are therefore triggered after every second unperturbed iteration.

The results of our calculations are given in Table \ref{tab:AcadSPM}. The unperturbed method exhibits comparably slow convergence speed due to the opposing projection steps, which offer little progress in the direction of the $x_3$-axis. Both the heavy ball and the surrogate constraint perturbation are able to speed up the iteration process significantly.

\begin{table}
\center
\caption{Iterations needed to find a feasible solution with $\lambda(x) \equiv 1.9$. SP converges slowly. Using perturbations accelerates the convergence significantly.}
\label{tab:AcadSPM}
\begin{tabular}{rccccc}
\hline\noalign{\smallskip}
& SP & \multicolumn{3}{c}{SP + HB} & SP + SC \\
& & $\lambda^{HB}_k =8$& $\lambda^{HB}_k =80$ & $\lambda^{HB}_k =800$& \\
\noalign{\smallskip}\hline\noalign{\smallskip}
K & 449 &58& 17& 4	&4\\
\noalign{\smallskip}\hline
\end{tabular}
\end{table}

Next, we solve \eqref{eq:AcademicExample} using the cyclic projection method with $\lambda(x)\equiv 1$ and the control sequence $\{j_1(\nu)\}_{\nu=0}^{\infty} = \{1, 2, 3, 4, 1,2,3,4,...\}$. This results in a sequence $\{p(x^k)\}$ of projection steps, which does not fulfill $\langle \bar{p}(x^k), \bar{p}(x^{k-1}) \rangle \in [-1 + \epsilon_{min}, -1+\epsilon_{max}]$ until the constraint violations fall below the tolerance of $10^{-10}$. Perturbations as defined in \eqref{eq:HBinnerPerturbationDef} or \eqref{eq:SCinnerPertDef} can therefore not be used in a meaningful way with this method and choice of parameters. The values of $\langle \bar{p}(x^k), \bar{p}(x^{k-1}) \rangle$ do not reflect the conflict inherent in the system of linear inequalities due to the choice of the control sequence.
It takes the unperturbed algorithm 1917 iterations to find a feasible solution $x^{\ast}$.

We now use the same control sequence in combination with $\lambda(x) \equiv 1.9$. This results in a sequence of iterates $\{x^k\}$ and projection steps  $\{p(x^k)\}$, which fulfill $\langle \bar{p}(x^k), \bar{p}(x^{k-1}) \rangle \in [-1 + \epsilon_{min}, -1+\epsilon_{max}]$ for some $k$. Therefore, we are able to use perturbations with this method and choice of parameters. The results of our calculations are presented in Table \ref{tab:AcadCPM_seq1}. This method with this particular choice of parameters is the quickest among our experiments to converge. We observe that for this choice of parameters, not all step sizes $\lambda_k^{HB}$ result in accelerated convergence behavior. Cyclic projection using the surrogate constraint perturbation, however, converges faster than any of the other variants of cyclic projection.

\begin{table}
\center
\caption{Iterations needed to find a feasible solution with $\lambda(x) \equiv 1.9$ and the control sequence $\{j_1(\nu)\}$. Not all perturbed versions of CP converge faster than unperturbed CP.}
\label{tab:AcadCPM_seq1}
\begin{tabular}{rccccc}
\hline\noalign{\smallskip}
& CP & \multicolumn{3}{c}{CP + HB} & CP + SC \\
& & $\lambda^{HB}_k =8$& $\lambda^{HB}_k =80$ & $\lambda^{HB}_k =800$& \\
\noalign{\smallskip}\hline\noalign{\smallskip}
K & 20 & 34&26&9 &4\\
\noalign{\smallskip}\hline
\end{tabular}
\end{table}

Finally, we alter the feasibility problem \eqref{eq:AcademicExample} by enlarging the linear inequality system with duplicates of the matrix rows $a_i$ and the values $b_i$ of the right hand side vector. We make this alteration in order to be able to give a control sequence, which both fits Definition \ref{eq:control} and results in the phenomena we will describe in the following. The new matrix $\tilde{A}\in\mathbb{R}^{8\times 3}$ and vector $\tilde{b}\in\mathbb{R}^8$ are given by

\begin{align}
\tilde{A} &\colonequals \left(a_1^T, a_3^T, a_1^T, a_3^T, a_2^T, a_4^T,a_2^T, a_4^T\right)^T \nonumber\\
\tilde{b}&\colonequals (b_1, b_3, b_1, b_3, b_2, b_4, b_2, b_4)^T \nonumber
\end{align}

The system of linear inequalities $\tilde{A}x\leq \tilde{b}$ of course has the same set of solutions as $Ax\leq b$.
We solve  $\tilde{A}x\leq \tilde{b}$ using the cyclic projection method with $\lambda = 1.9$ and the control sequence $\{j_2(\nu)\}_{\nu=0}^{\infty} =\{1, 2, 3, 4, 5, 6, 7, 8, 1, 2, 3, 4, 5, 6, 7, 8, ...\}$. By repeating pairs of almost opposing inequalities, we generate a sequence of projection steps $\{p(x^k)\}$, which fulfills  $\langle\bar{p}(x^k), \bar{p}(x^{k-1}) \rangle \in [-1 + \epsilon_{min}, -1+\epsilon_{max}]$ for some $k$. The number of iterations needed by the variants of the cyclic projection to find a feasible solution are given in Table \ref{tab:AcadCPM_seq2}. Here, all perturbations are able to accelerate the convergence. Again, cyclic projection using the surrogate constraint perturbation was the fastest to converge. In Figure \ref{fig:acadexcpmethodcomparison} we present the norm of the constraint violation $\|\max\{0, Ax^k-b\}\|$ plotted against the iteration index $k$. HB8, HB80 and HB800 denote the three versions of the heavy ball perturbation using $\lambda_k^{HB} = 8, 80$ or $800$.

\begin{table}
\center
	\caption{Iterations needed to find a feasible solution with $\lambda(x) \equiv 1.9$ and the control sequence $\{j_2(\nu)\}$. Perturbed CP converges faster than unperturbed CP.}
	\label{tab:AcadCPM_seq2}
	\begin{tabular}{rccccc}
		\hline\noalign{\smallskip}
		& CP & \multicolumn{3}{c}{CP + HB} & CP + SC \\
		& & $\lambda^{HB}_k =8$& $\lambda^{HB}_k =80$ & $\lambda^{HB}_k =800$& \\
		\noalign{\smallskip}\hline\noalign{\smallskip}
		K & 32 & 29&20&7 &3\\
		\noalign{\smallskip}\hline
	\end{tabular}
\end{table}

\begin{figure}
\center
	\includegraphics[width=0.75\textwidth]{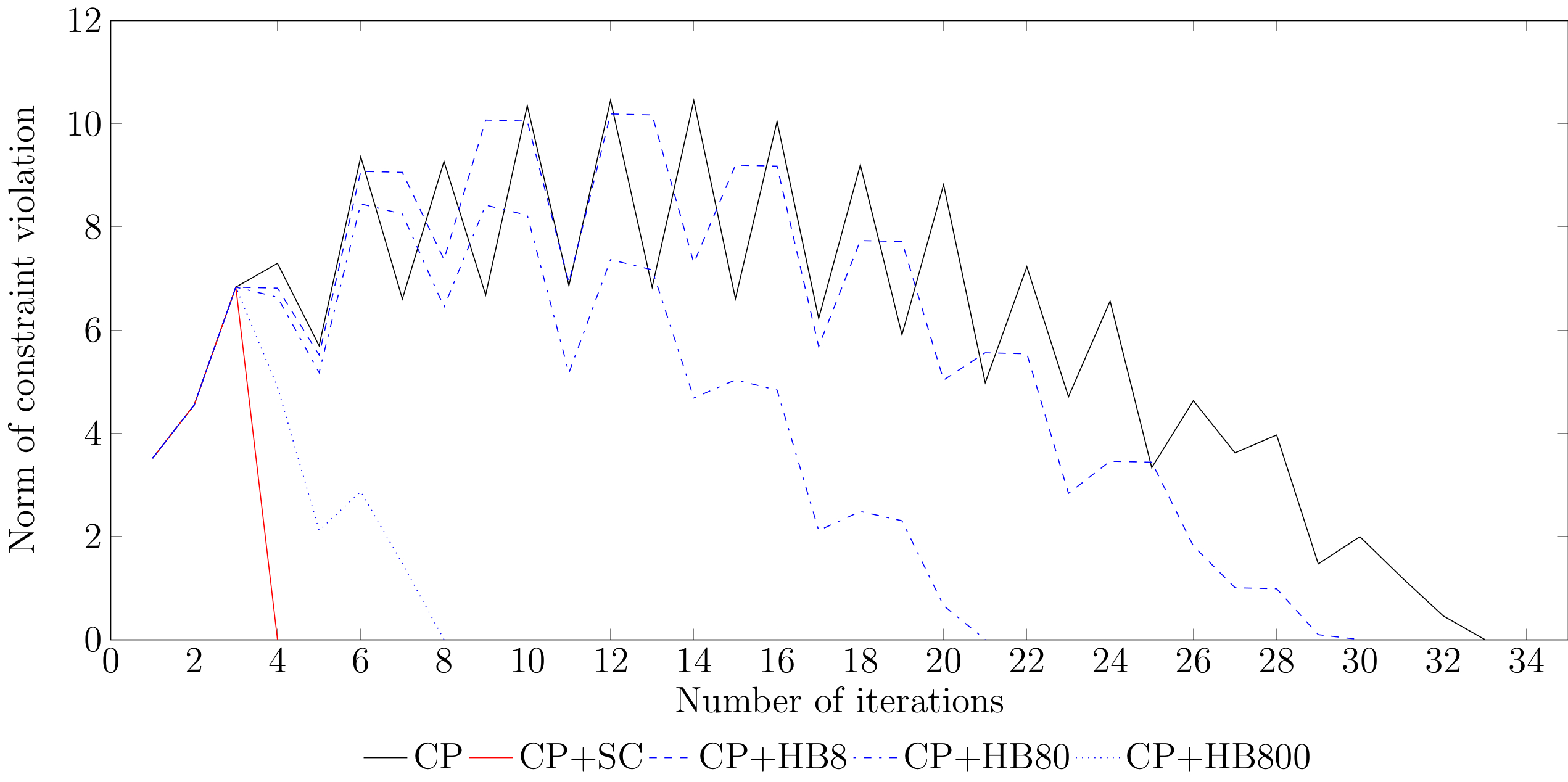}
	\caption{Comparison of the constrain violation norm $\|\max\{0, Ax^k-b\}\|$ for perturbed and unperturbed CP. CP using SC perturbation is the fastest and unperturbed CP is the slowest to converge to a feasible solution.}
	\label{fig:acadexcpmethodcomparison}
\end{figure}

\subsection{Intensity-modulated radiation therapy}

Intensity-modulated radiation therapy (IMRT) is used as one treatment option in clinical oncology. With it, physicians aim to destroy the tumor tissue with irradiation while sparing surrounding healthy organs as much as possible.
These goals are conflicting because of physical limitations like the proximity of the tumor to healthy organs and the attainable decrease of the dose values over a given distance. Therefore, the treatment planner usually has to decide for a compromise between several goals.

The IMRT planning problem, as we consider it in this paper, is to determine optimal intensity maps, i.e. a set of fluence intensity values, which causes a dose distribution in the patient's body which will best fulfill the clinical goals formulated by the physician. In contrast to the previous examples, the IMRT planning problem is a nonlinear one.

\subsubsection{Tools for modeling}

The vector $d$ of dose values received by each voxel in the patient's body when the fluence intensity given by the vector $x$ is applied can be calculated using the so-called dose matrix $P$. Due to physical reasons, the vector $x$ is restricted to be non-negative, which results in dose values $d= P \cdot x$, which are also non-negative.

We use the following set of functions to represent the dose prescriptions given by the clinical goals. All functions refer to a biological structure and evaluate the dose received by the voxels contained in the structure. From a mathematical point of view that structure is a set $\mathcal{O}$ of indices which correspond to the voxels contained in the biological structure. By $d_i = \langle p_i, x\rangle$ we denote the dose value received by the voxel with index $i$, where $p_i$ is the $i$-th row of $P$.

\paragraph{Quadratic upper tail penalty function}

The upper tail penalty function penalizes dose values of $d$ corresponding to the structure $\mathcal{O}$ which exceed a given threshold $U\in\mathbb{R}$.

$$f(d, \mathcal{O}) = \frac{1}{|\mathcal{O}|}\sum_{i\in \mathcal{O}} \max(0, d_i-U)^2 $$

\paragraph{Quadratic lower tail penalty function}

The lower tail penalty function penalizes dose values of $d$ corresponding to the structure $\mathcal{O}$ which fall below a given threshold $L\in\mathbb{R}$.

$$f(d, \mathcal{O}) = \frac{1}{|\mathcal{O}|}\sum_{i\in \mathcal{O}} \max(0, L-d_i)^2 $$

\paragraph{Equivalent uniform dose (EUD)}
The EUD function is related to the well-known EUD concept of Niemierko \cite{niemierko1997}. It penalizes dose values of $d$ corresponding to the structure $\mathcal{O}$ which deviate from 0.

$$f(d, \mathcal{O}) = \frac{1}{|\mathcal{O}|}\sum_{i\in \mathcal{O}} d_i^p $$

\paragraph{Tumor conformity}
The tumor conformity function is used to ensure an even dose distribution within the tumor volume. It penalizes dose values of $d$ corresponding to the structure $\mathcal{O}$ which deviate from a given reference value.

$$f(d, \mathcal{O}, d^{ref}) = \frac{1}{|\mathcal{O}|}\sum_{i\in \mathcal{O}} |d^{ref}-d_i|^p$$

\begin{figure}
\center
	\includegraphics[width=0.75\textwidth]{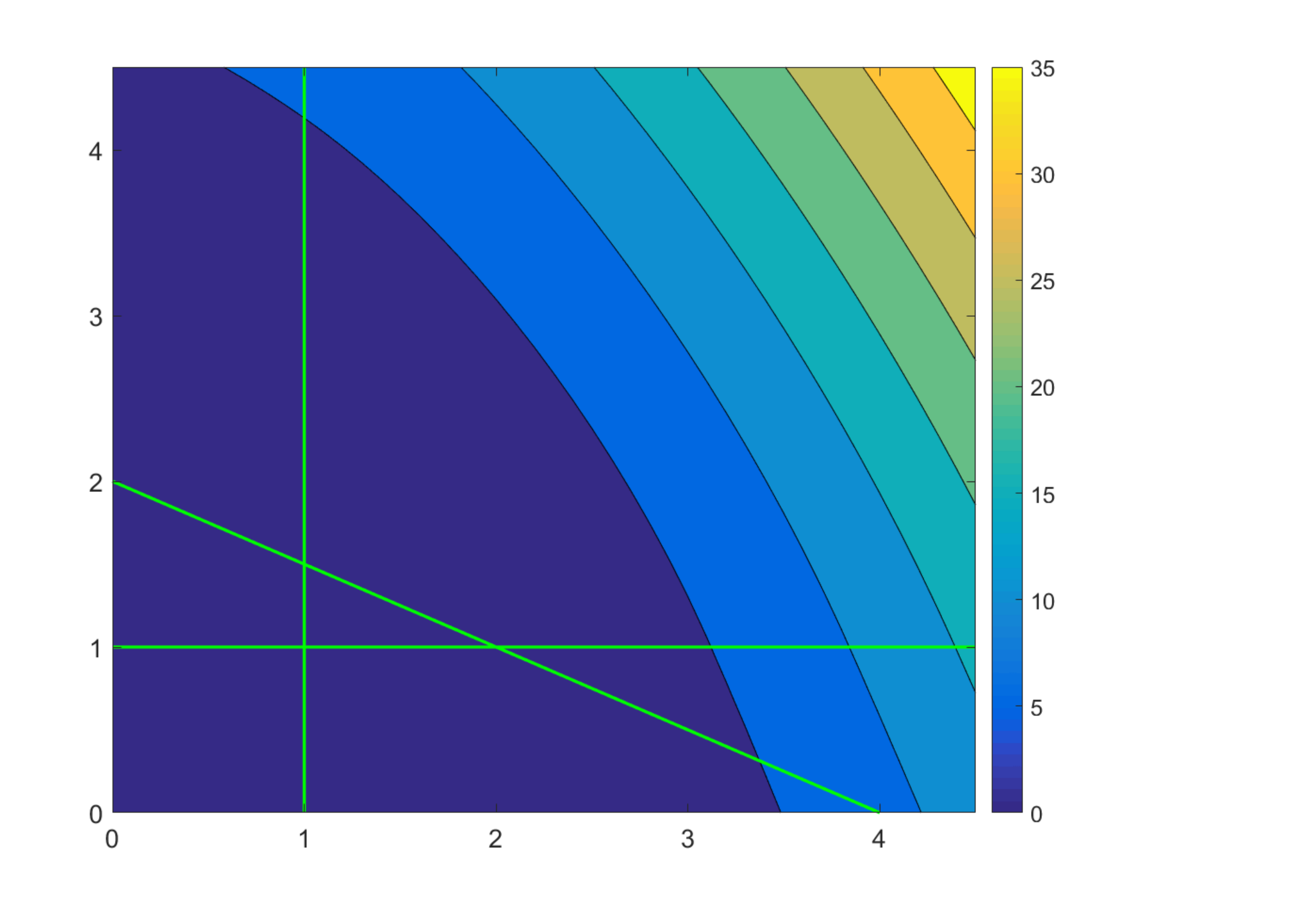}
	\caption{2D example for the upper tail penalty function. The green lines represent the hyperplanes separating feasible (towards the origin) from infeasible half-spaces. Isolines of the upper tail penalty function values are given according to the color map.}
	\label{fig:tailpenaltyscaledsquare}
\end{figure}

All of these functions take into account the system of linear (in)equalities resulting from $d_i = \langle p_i, x\rangle$ being less/ greater or equal than a right hand side value determined by the function parameters, where $i\in\mathcal{O}$. The functions measure the distance of the vector $x$ to the (separating) hyperplanes corresponding to the linear (in)equalities in a nonlinear way. Figure \ref{fig:tailpenaltyscaledsquare} illustrates this concept with a simple 2D example.

The given functions do not represent a complete list of dose evaluation functions used in clinical applications. A more complete survey of such functions can be found in \cite{rdl04} and \cite{sfom99}.
\subsubsection{The IMRT optimization problem}

In this work we choose four head neck cases for our numerical experiments. We consider a reduced set of biological structures and their dose evaluation functions, which focuses on the main conflict between irradiating the tumor volume on the one hand and sparing the myelon and parotids on the other hand. Additionally, we include the healthy tissue not associated with any of the structures mentioned above with the goal to keep the dose in this tissue as low as possible. The dose evaluation functions we use are given in Table \ref{tab:IMRTmodel} in detail.

We use a reduced set of biological structures for our calculations, because our focus is to demonstrate the effects of the mathematical methods. Therefore, the treatment plans resulting from our calculations do not fulfill all of the clinical goals a treatment planner would formulate for a full head neck case.

\begin{table}
\center
\caption{Dose evaluation functions of the considered biological structures}
\label{tab:IMRTmodel}
\begin{tabular}{clcc}
\hline\noalign{\smallskip}
	& structure & function name & parameters\\
\noalign{\smallskip}\hline\noalign{\smallskip}
$f_1$ & Left parotis & EUD & $p=2$\\
$f_2$ & Right parotis & EUD & $p=2$\\
$f_3$ & Myelon & EUD & $p=2$\\
$f_4$ & Unclassified healthy tissue & EUD & $p=2$\\
$f_5$ & Tumor volume & Tumor conformity & $p=2, d^{ref} = 60$\\
$g_1$ & Tumor volume & lower tail penalty & $L = 55$\\
$g_2$ & Tumor volume & upper tail penalty & $U = 66$\\
$g_3$ & Myelon & upper tail penalty & $U = 45$\\
\noalign{\smallskip}\hline
\end{tabular}
\end{table}

The IMRT optimization problem as we formulate it in this work is

\begin{align} \label{eq:IMRToptproblem}
\text{Minimize } f(x)&\\
\text{s.t. }g_j(x)&\leq 0\quad j\in J\nonumber\\
x&\geq 0\nonumber
\end{align}
where $f=\sum_{i\in I} f_i$, $I =\{1,2,3,4, 5\}  , J = \{1, 2, 3\}$ and $f_i, g_j$ as listed in Table \ref{tab:IMRTmodel}.

\subsubsection{Results}
We use the combination of methods described in Section \ref{sec:Algorithm} to solve the problem \eqref{eq:IMRToptproblem} for four different head neck cases. Results of these calculations are presented in Tables \ref{tab:IMRTiterationNumbersAndObjValsSPM},  \ref{tab:IMRTiterationNumbersAndObjValsCPM}, \ref{tab:IMRTnumPerturbations} and \ref{tab:IMRTspeedupFactor} and, in an exemplary manner for one of the cases, in Figures \ref{fig:c1methodcomparison}, \ref{fig:c1pertindices} and \ref{fig:c1DHVs}.
For our calculations we choose $\lambda(x) \equiv 1.9, \lambda^{SC}_k  = 1, \lambda^{HB}_k = 1$ and  $\epsilon_{max} =0.034, \epsilon_{min} = 10^{-8}$. We stop the algorithm if the projection method we use is unable to find a feasible solution of the current CFP after $10^3$ iterations. Then, the current CFP is assumed to be infeasible and we consider the solution of the previous CFP to be the result of our algorithm.

By $f^{\ast}_{\text{method, perturbation}}$ we denote the lowest objective function value, for which the algorithm using this method and perturbation is able to find a solution within the given maximum number of iterations per CFP $\boldsymbol{P}^s$ as described in \eqref{Problem:k_CFP}.
In the same way we denote by $K_{\text{method, perturbation}}$ the total number of iterations it takes the specified method using the specified perturbation to find the solution $x^{\ast}_{\text{method, perturbation}}$ with the optimal objective function value $f^{\ast}_{\text{method, perturbation}}$.

Tables \ref{tab:IMRTiterationNumbersAndObjValsSPM} and \ref{tab:IMRTiterationNumbersAndObjValsCPM} present the lowest objective function values achieved by the simultaneous or the cyclic projection method in an unperturbed manner or using either the heavy ball perturbation or the surrogate constraint perturbation. The values are given with a precision of $10^0$. Furthermore, the number of iterations needed to find the solutions are presented. For easier comparison, the percentages with respect to the values corresponding to the unperturbed methods are given in brackets.

Note that the optimal objective function values of the perturbed simultaneous projection method $f^{\ast}_{\text{SP, HB}}$ and $f^{\ast}_{\text{SP, SC}}$ (given in Table \ref{tab:IMRTiterationNumbersAndObjValsSPM}) are smaller than the optimal objective function value $f^{\ast}_{\text{SP}}$ of the unperturbed method. The same is true for the cyclic projection method.

When we use cyclic projection, the condition required to apply the perturbations given in equations \eqref{eq:HBinnerPerturbationDef} and \eqref{eq:SCinnerPertDef} is never met for two of the four head neck cases. We denote the values for  $f^{\ast}_{\text{method, perturbation}}$ and $K_{\text{method, perturbation}}$ by "$-$".

We observe that for all cases $K_{\text{SP, HB}}, K_{\text{SP, SC}} < K_{\text{SP}}$, but for some cases $K_{\text{CP, HB}}$,  $K_{\text{CP, SC}} > K_{\text{CP}}$. This means that in these cases, the perturbed method continues beyond solutions with the objective value $f^{\ast}_{\text{CP}}$, but in total takes more iterations than $K_{\text{CP}}$ to get there.

To answer the question, whether perturbed methods converge faster, measured at the same objective function value, we present Table \ref{tab:IMRTspeedupFactor}. There we give the number of iterations it takes the perturbed methods to find a solution with an objective function value less or equal than $f^{\ast}_{\text{SP}}$ or $f^{\ast}_{\text{CP}}$, normalized by $K_{\text{SP}}$ and  $K_{\text{CP}}$. These values are given with a precision of $10^{-4}$ and are an indicator for the acceleration of the iteration process caused by the perturbations we used. Note that in all cases the surrogate constraint perturbation is able to speed up the iteration process even more than the heavy ball perturbation.

Figure \ref{fig:c1methodcomparison} illustrates the progress of the different methods in an exemplary way for Case 1. It is notable that perturbations are used rather early in the iteration process when we use simultaneous projection and comparably late when we use cyclic projection. The indices of iterations in which perturbations are used by the different methods are illustrated in more detail by Figure \ref{fig:c1pertindices} for the same case as in Figure \ref{fig:c1methodcomparison}.

This phenomenon occurs due to the fact that simultaneous projection uses a weighted sum of all function gradients corresponding to violated constraints, whereas cyclic projection uses only gradient information of the next (with respect to the control sequence) violated constraint function.

In our model, there are groups of functions which correspond to conflicting goals. By summing the function gradients, simultaneous projection incorporates the information about the conflict between these groups of functions from a very early stage of the iteration process. While this results in almost opposing subsequent projection steps $p(x^{k-1}), p(x^k)$, which leads to slow convergence, it also triggers the perturbation of the iteration process.
When we use cyclic projection, the conflict between the groups of functions only becomes obvious when the subsequently (with respect to the control sequence) violated constraints have opposing function gradients.
Figure \ref{fig:c1pertindices} shows that this happens for IMRT cases rather late in the iteration process. The values in Table \ref{tab:IMRTspeedupFactor} indicate that methods using perturbations early converge faster than methods using them later in the iteration process.

Table \ref{tab:IMRTnumPerturbations} presents the number of iterations, in which perturbations are used. We observe that in all cases the heavy ball perturbation is applied more often than the surrogate constraint perturbation. Together with $K_{\text{SP, SC}} < K_{\text{SP, HB}}$ and $K_{\text{CP, SC}} < K_{\text{CP, HB}}$, this indicates that in our computations the surrogate constraint perturbation is more effective than the heavy ball perturbation.

Finally, in Figure \ref{fig:c1DHVs}, we present a cumulative dose volume histogram (DVH) in which we compare the dose distributions calculated for case 1 by simultaneous projection without perturbation and with the surrogate constraint perturbation. DVHs are a tool used by treatment planners to evaluate the quality of a fluence map and the resulting dose distribution in the patient's body. DVHs show, which percentage of the volume of a certain structure receives a dose greater or equal than the dose value on the horizontal axis. For the tumor volume, a treatment planner might want that 95\% of the volume receives at least a dose of 55 Gy and for the myelon, they might want that at most 5\% of the volume receives a dose greater than 45 Gy.

The solid lines in Figure \ref{fig:c1DHVs} represent the solution resulting from simultaneous projection without perturbation and the dashed lines correspond to the solution resulting from simultaneous projection with the surrogate constraint perturbation. We observe that the curves for the tumor volume and both parotids do not differ much, but the curves for the myelon are significantly lower for the perturbed method and therefore represent a more desirable dose distribution than the distribution resulting from the unperturbed method.

\begin{table}
\center
\caption{Iteration numbers and lowest objective function values achieved using perturbed and unperturbed SP. Perturbed SP reaches solutions with lower objective function values within less iterations.}
\label{tab:IMRTiterationNumbersAndObjValsSPM}
\begin{tabular}{rcccccc}
\hline\noalign{\smallskip}
& $f^{\ast}_{\text{SP}}$ &  $f^{\ast}_{\text{SP, HB}}$ &  $f^{\ast}_{\text{SP, SC}}$& $K_{\text{SP}}$ & $K_{\text{SP, HB}}$ & $K_{\text{SP, SC}}$ \\
\noalign{\smallskip}\hline\noalign{\smallskip}
Case 1 &3480&3464 (99.54\%)&3387 (97.33\%)	&7159&3437 (48.00\%)&2155 (30.10\%)\\
Case 2 &2378&2356 (99.07\%)&2317 (97.43\%)	&3523&2340 (66.42\%)&1108 (31.45\%)\\
Case 3 &3129&3056 (97.67\%)&3012 (96.26\%)	&4773&3458 (72.45\%)&1171 (24.53\%)\\
Case 4 &3098&2980 (96.19\%)&2941 (94.93\%)	&4496&3898 (86.70\%)&1292 (28.74\%)\\
\noalign{\smallskip}\hline
\end{tabular}
\end{table}

\begin{table}
\center
\caption{Iteration numbers and lowest objective function values achieved using perturbed and unperturbed CP. In only two of four cases perturbations are used and achieve lower objective function values than unperturbed CP.}
\label{tab:IMRTiterationNumbersAndObjValsCPM}
\begin{tabular}{rccccccc}
\hline\noalign{\smallskip}
& $f^{\ast}_{\text{CP}}$ &  $f^{\ast}_{\text{CP, HB}}$ &  $f^{\ast}_{\text{CP, SC}}$& $K_{\text{CP}}$ & $K_{\text{CP, HB}}$ & $K_{\text{CP, SC}}$ \\
\noalign{\smallskip}\hline\noalign{\smallskip}
	Case 1 &3563&3490 (97.95\%)&3420 (95.99\%)	&6665&6103 (91.57\%)&4044 (60.66\%)\\
Case 2 &2424&-&-	&4484&-&-\\
Case 3 &3178&3083 (97.01\%)&3052 (96.04\%)	&4238&5929 (139.90\%)&5035 (118.81\%)\\
Case 4 &3093&-&-&	5280&-&-\\
\noalign{\smallskip}\hline
\end{tabular}
\end{table}

\begin{table}
\center
\caption{The number of perturbations used by the perturbed projection methods. For all cases HB perturbation is used more often than SC perturbation.}
\label{tab:IMRTnumPerturbations}
\begin{tabular}{rcccc}
\hline\noalign{\smallskip}
& SP + HB & SP + SC & CP + HB & CP + SC\\
\noalign{\smallskip}\hline\noalign{\smallskip}
Case 1 &803&283	&789&374\\
Case 2 &534&101	&0&0\\
Case 3 &922&166	&451&349\\
Case 4 &1074&180	&0&0\\
\noalign{\smallskip}\hline
\end{tabular}
\end{table}

\begin{table}
\center
\caption{The fraction of $K_{\text{SP}}$ or $K_{\text{CP}}$ needed by the perturbed methods to find a solution with objective function value $\leq f^{\ast}_{\text{SP}}$ or $f^{\ast}_{\text{CP}}$. SC perturbation achieves lower values than HB perturbation, in particular if the perturbed method is SP.}
\label{tab:IMRTspeedupFactor}      
\begin{tabular}{rcccc}
\hline\noalign{\smallskip}
& SP + HB & SP + SC & CP + HB & CP + SC\\
\noalign{\smallskip}\hline\noalign{\smallskip}
Case 1 &0.4801&0.2122	&0.7685&0.4308\\
Case 2 &0.6642&0.2191	&-&-\\
Case 3 &0.4324&0.1402	&0.9122&0.8554\\
Case 4 &0.4121&0.1417	&-&-\\
\noalign{\smallskip}\hline
\end{tabular}
\end{table}

\begin{figure}
\center
	\includegraphics[width=0.75\textwidth]{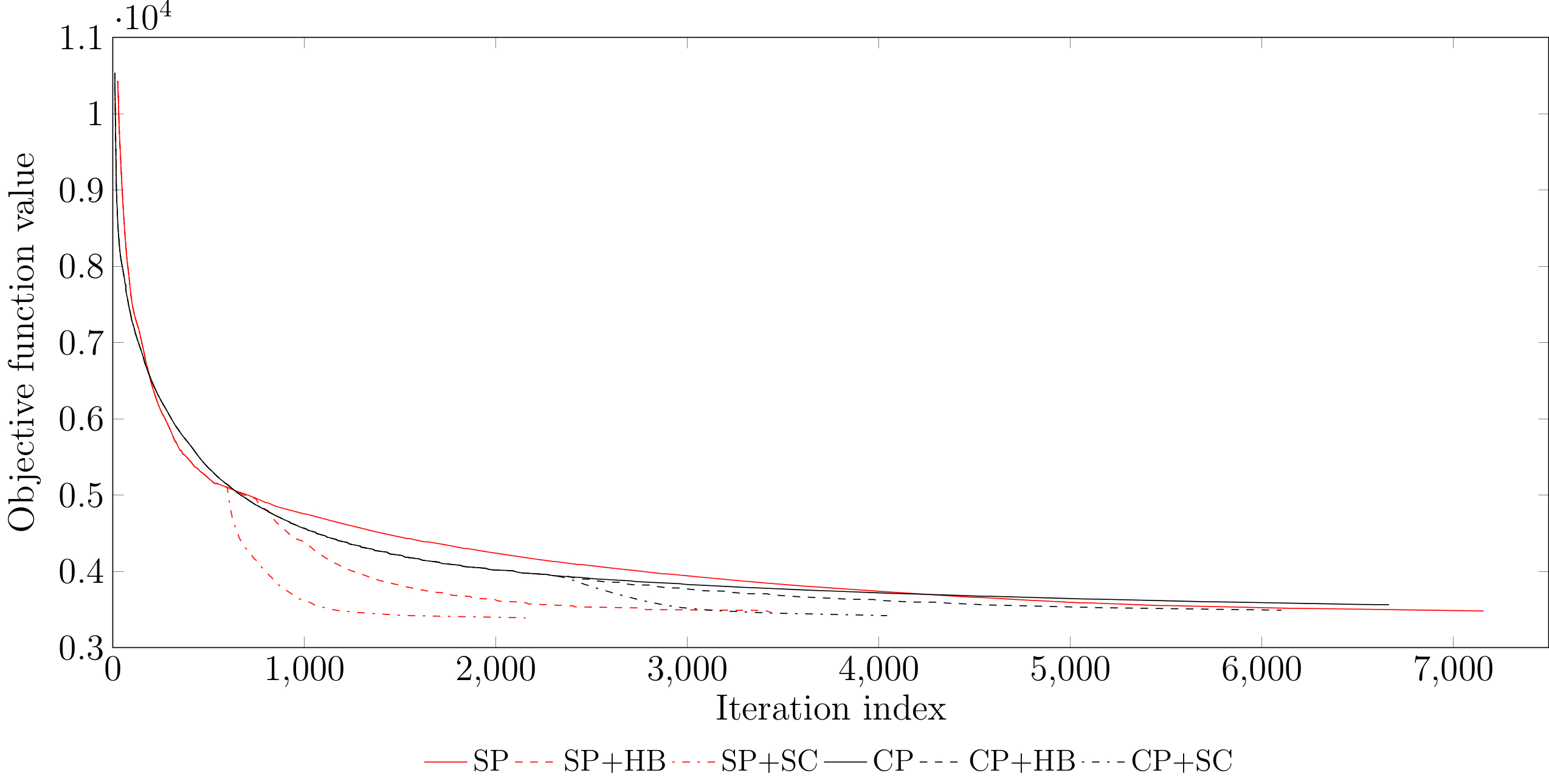}
	\caption{Objective function values achieved by different methods plotted against the required number of iterations. SP using SC perturbation achieves the lowest value and is the quickest to converge. Unperturbed CP is the slowest method and produces the solution with the highest objective function value.}
	\label{fig:c1methodcomparison}
\end{figure}

\begin{figure}
\center
	\includegraphics[width=0.75\textwidth]{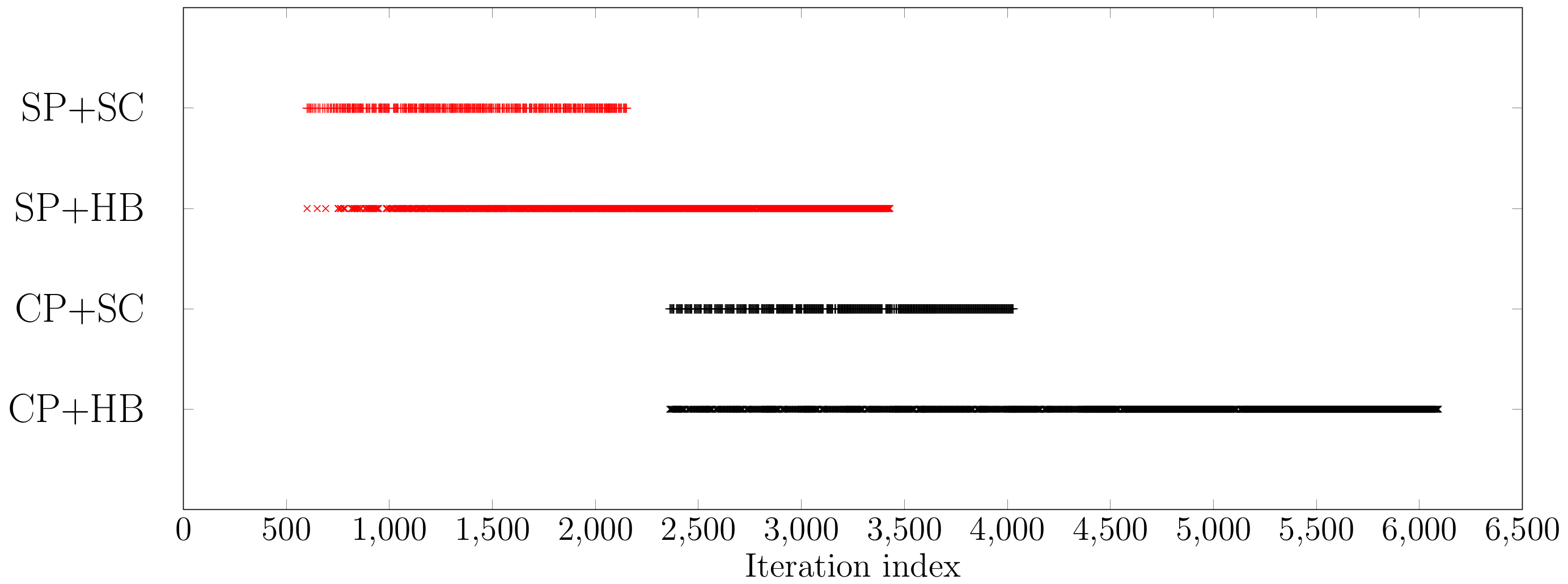}
	\caption{Indices of perturbed iterations. SP uses perturbations earlier than CP.}
	\label{fig:c1pertindices}
\end{figure}

\begin{figure}
\center
	\includegraphics[width=0.75\textwidth]{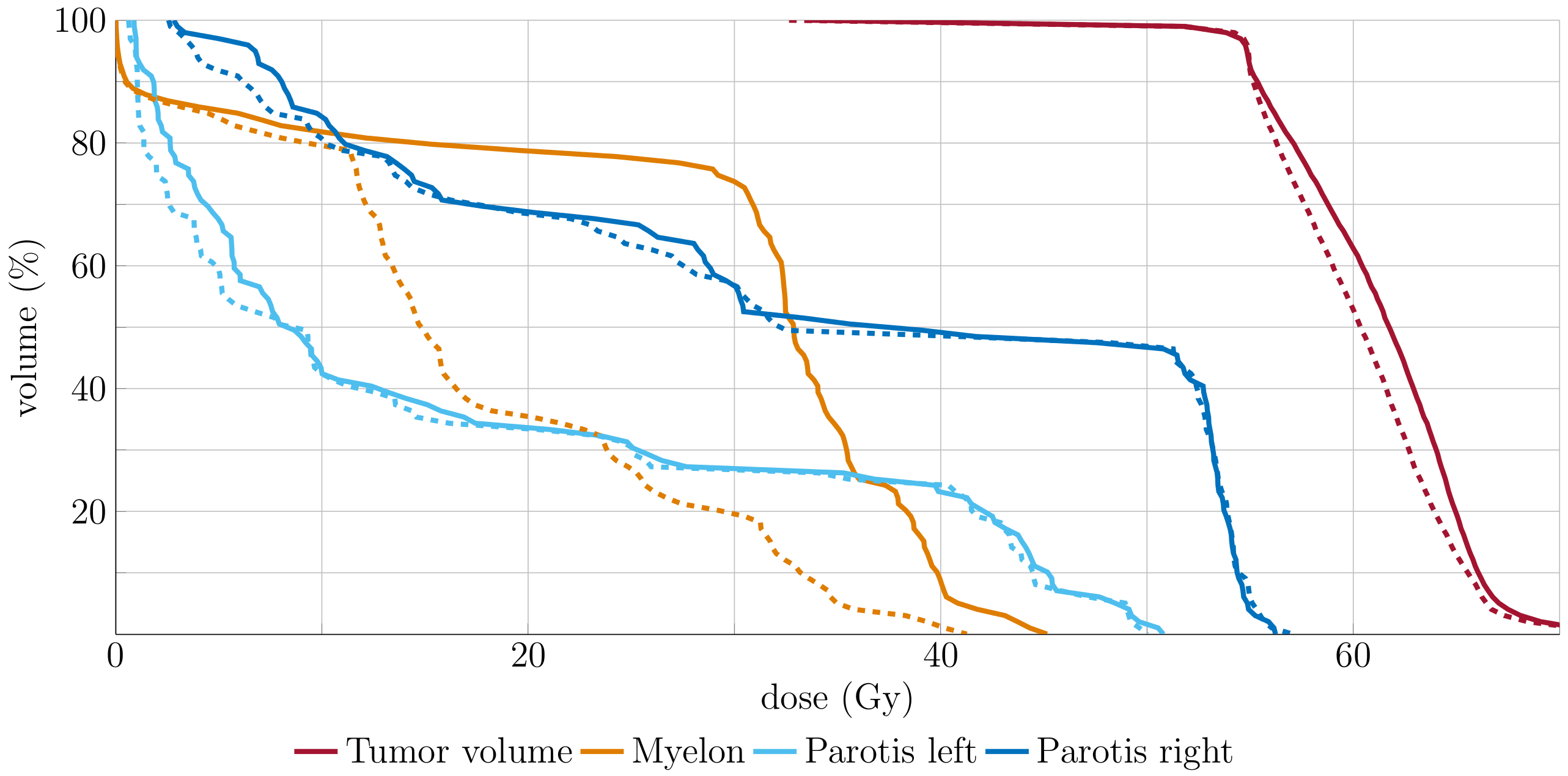}
	\caption{Dose volume histogram resulting from solution doses produced by SP (solid lines) and SP using SC perturbations (dashed lines).}
	\label{fig:c1DHVs}
\end{figure}

\section{Conclusions}\label{sec:Conclusions}

In this paper we transformed a general constrained convex optimization problem into a sequence of feasibility problems via the level set scheme. We solved each feasibility problem using either the simultaneous or cyclic subgradient projection method. We exploited the fact that both projection methods are bounded perturbation resilient and modified the iteration scheme of the projection methods with the heavy ball and the surrogate constraint perturbation.

Our numerical results demonstrate that the perturbed projection methods converge faster than their unperturbed counterparts, both for the linear feasibility problems we discussed and those arising from IMRT treatment planning.
Applying perturbations early in the iteration process, as it was done for the simultaneous projection, yields better results than applying them later, like for the cyclic projection.

In our computations the perturbed versions of the simultaneous projection method offer the biggest improvement in both iteration numbers and objective function values of the solutions. The simultaneous projection method using the surrogate constraint perturbation surpasses the objective function value achieved by the unperturbed simultaneous projection method after only 14-22\% of the number of iterations needed by the unperturbed method. The least improvement is achieved by the cyclic projection using the heavy ball perturbation. This method surpasses the objective value achieved by the unperturbed cyclic projection after 43-91\% of the iterations needed by the unperturbed method.

In our calculations the surrogate constraint perturbation outperforms the heavy ball perturbation every time. Initial experiments indicate that this discrepancy results from the choice of the step length $\lambda^{HB}$. More detailed results are left for future publications. In our problem setting the computational cost for the surrogate constraint perturbation is slightly higher than for the heavy ball perturbation (12$n$+1 FLOPS vs. 11$n$ FLOPS, where $n$ is the length of the decision variable vector $x\in\mathbb{R}^n$). In problem settings where this difference is critical the heavy ball perturbation might be a more attractive approach.

Our observations suggest that both perturbations introduced by us can be used as an acceleration technique when zigzagging behavior occurs. Hence these and other perturbations should be further investigated.\\

\section*{Acknowledgment}
We wish to thank the editor and the anonymous referees for the thorough analysis and review, their comments and suggestions helped tremendously in improving the quality of this paper and made it suitable for publication.

\noindent
\textbf{\underline{Conflict of Interest:}} The authors declare that they have no conflict of interest.


\begin{thebibliography}{}


\bibitem{acpr:18}H.~Attouch, Chbani Z., J.~Peypouquet, P.~Redont, Fast convergence of inertial dynamics and algorithms with asymptotic vanishing damping, Math. Program. 168, 123--175 (2018)

\bibitem{brz17}C. Bargetz, S. Reich, R. Zalas, Convergence Properties of Dynamic String
Averaging Projection Methods in the Presence of Perturbations, Numer. Algorithms 77(1), 185--209 (2017)

\bibitem {Bertsekas99}D. P. Bertsekas, Nonlinear Programming: 2nd
	Edition, Athena Scientific, Belmont, Massachusetts, USA (1999)

\bibitem {bgks17}  E. Bonacker, A. Gibali, K.-H. K\"{u}fer, P. S\"{u}ss, Speedup of lexicographic optimization by superiorization and its applications to cancer radiotherapy treatment, Inverse Problems 33(4), 044012 (2017)

\bibitem {Borwein:2018aa}J.M. Borwein, S.B. Lindstrom, B.~Sims, A.~Schneider, M.P. Skerrit, Dynamics of the Douglas-Rachford method for ellipses and $p$-spheres, Set-Valued Var. Anal. 26, 385--403 (2018)

\bibitem {bdhk07} D. Butnariu, R. Davidi, G.T. Herman, I.G. Kazantsev, Stable convergence behavior under summable perturbations of a class of projection methods for convex feasibility and optimization problems, IEEE Journal of Selected Topics in Signal Processing 1, 540--547 (2007)

\bibitem{cegielski13}A.Cegielski, Iterative Methods for Fixed Point Problems in Hilbert Spaces, Springer-Verlag Berlin Heidelberg (2013)

\bibitem {censor18}Y. Censor, Superiorization and perturbation
resilience of algorithms: A bibliography compiled and continuously updated.
https://arxiv.org/abs/1506.04219. \newline Online at:
http://math.haifa.ac.il/yair/bib-superiorization-censor.html (last updated:
May 27, 2018.)	

\bibitem {censor15}Y. Censor, Weak and strong superiorization:
Between feasibility-seeking and minimization, Analele Stiintifice ale
	Universitatii Ovidius Constanta-Seria Matematica 23, 41--54 (2015)

\bibitem {censor17}Y. Censor, Can linear superiorization be useful for
linear optimization problems?, Inverse Problems 33, 044006 (2017)

\bibitem{ccp11}Y. Censor, W. Chen, H. Pajoohesh, Finite convergence of a subgradient
projections method with expanding controls, Appl. Math. Optim. 64, 273–285 (2011)

\bibitem {cdh10}Y. Censor, R. Davidi, G.T. Herman, Perturbation resilience
and superiorization of iterative algorithms, Inverse Problems 26, 065008 (2010)

\bibitem {cdhst14}Y. Censor, R. Davidi, G.T. Herman, R.W. Schulte, L. Tetruashvili, Projected subgradient minimization versus superiorization, J. Optim. Theory Appl. 160, 730--747 (2014)

\bibitem{ceh2000} Y. Censor, T. Elfving, G. T. Herman, Averaging strings of sequential iterations for convex
feasibility problems. In: D. Butnariu, Y. Censor, S. Reich (eds.) Studies in Computational Mathematics 8, pp. 101--113, Elsevier, North Holland, Amsterdam (2001)

\bibitem{special-issue}Y. Censor, G.T. Herman, M. Jiang (Editors),
Superiorization: Theory and Applications, Inverse Problems 33, Special Issue (2017)

\bibitem{CensorReem2014} Y. Censor, D. Reem, Zero-convex functions, perturbation resilience, and subgradient projections for feasibility-seeking methods, Math. Program. 152(1), 339--380 (2015)

\bibitem {cz15}Y. Censor, A.J. Zaslavski, Strict Fej\'{e}r monotonicity
by superiorization of feasibility-seeking projection methods, J. Optim. Theory Appl. 165, 172--187 (2015)
		
\bibitem{cz97}Y. Censor, S.A. Zenios, Parallel Optimization: Theory, Algorithms, and Applications, Oxford University Press, New York, New York, USA (1997)	

\bibitem{combettes97}P.L. Combettes, Hilbertian convex feasibility problem: convergence of projection methods, Appl. Math. Optim. 35, 311–330 (1997)

\bibitem {combettes01}P. L. Combettes, On the numerical robustness of the parallel projection method in signal synthesis, IEEE Signal Processing Letters 8(2), 45–47 (2001)

\bibitem {dgj18}Q.-L. Dong, A.~Gibali, D.~Jiang, S.-H. Ke, Convergence of projection and contraction algorithms with outer perturbations and their applications to sparse signals recovery, J. Fixed Point Theory Appl. 16, 16 (2018)

\bibitem {dgjt17}Q.-L. Dong, A.~Gibali, D.~Jiang, Y.~Tang, Bounded perturbation resilience of extragradient-type methods and their applications, J. Inequal. Appl 2017(1), 280 (2017)

\bibitem{Dudek2007} R. Dudek, Iterative Method for Solving the Linear Feasibility Problem, J. Optim. Theory Appl. 132(3), 401--410 (2007)

\bibitem {frv:18}G. Franca, D. P. Robinson, R. Vidal, Admm and accelerated admm as continuous dynamical systems. In: Dy, J., Krause, A. (eds.) Proceedings of the 35th International Conference on Machine Learning 80, pp. 1559--1567, PMLR, Stockholm Sweden (2018)

\bibitem{gh14} E. Gard\~{n}uo, G.T. Herman, Superiorization of the ML-EM algorithm, IEEE Transactions on Nuclear Science 61, 162--172 (2014)

\bibitem{gkrs18}A. Gibali, K.-H. K\"{u}fer, D. Reem, P. S\"{u}ss, A generalized projection-based scheme for solving
convex constrained optimization problems, Comput. Optim. Appl. 70(3), 737-762 (2018)

\bibitem{jcj13}W. Jin, Y. Censor, M. Jiang, A heuristic superiorization-like approach to bioluminescence, International Federation for Medical and Biological Engineering (IFMBE) Proceedings 39, 1026-1029 (2013)

\bibitem {Nesterov:83}Y.~E. Nesterov, A method for solving the convex programming problem with convergence rate o($1/k^{2}$), Doklady Akademii Nauk SSSR 269, 543--547 (1983)

\bibitem{niemierko1997}A. Niemierko, Reporting and analyzing dose distributions: a concept of equivalent uniform dose, Med. Phys. 24, 103--110 (1997)

\bibitem{pi1988}A. De Pierro, A. Iusem, A finitely convergent ``row-action'' method for the convex feasibility problem, Appl. Math. Optim. 17(1), 225 -- 235 (1988)

\bibitem {Polyak:64}B.~T. Polyak, Some mehtods of speeding up the convergence of iteration methods, USSR Computational Mathematics and Mathematical Physics 5, 1--17 (1964)

\bibitem{rdl04}H. Romeijn, J. Dempsey, J. Li, A unifying framework for multicriteria fluence map optimization models, Phys. Med. Biol. 49, 1991--2013 (2004)

\bibitem{sv12}S. Salzo, S. Villa, Inexact and accelerated proximal point algorithms, J. Convex Anal. 19, 1167--1192 (2012)

\bibitem{sfom99}D.M. Shepard, M.C. Ferris, G.H. Olivera, T.R. Mackie, Optimizing the delivery of radiation therapy to cancer patients, SIAM Review 419, 721--744 (1999)

\bibitem{tw13}C. Tian, F. Wang, The contraction-proximal point algorithm with square-summable errors, J. Fixed Point Theory Appl. 2013(93) (2013). https://doi.org/10.1186/1687-1812-2013-93

\bibitem{wang2008} X. Wang, Method of steepest descent and its applications, IEEE Microwave Wireless Components Letters 12, 24--26 (2008)

\bibitem{Yang1992} K. Yang, K. G. Murty, New iterative methods for linear inequalities, J. Optim. Theory Appl. 72(1), 163--185 (1992)

\bibitem{znly18}X. Zhao, K. F. Ng, C. Li, J.-C. Yao, Linear Regularity and Linear Convergence of Projection-Based Methods for Solving Convex Feasibility Problems, Appl. Math. Optim. 78(3), 613--641 (2018)
\end{thebibliography}
\end{document}